\documentclass[12pt]{amsart}
\usepackage{amssymb}
\usepackage{epsfig}
\numberwithin{equation}{section}

\input xy
\xyoption{all}    

\theoremstyle{plain}
\newtheorem{prop}[subsection]{Proposition}
\newtheorem{Theo}[section]{Theorem}
\newtheorem{thm}[subsection]{Theorem}
\newtheorem{coro}[subsection]{Corollary}

\newtheorem{lemm}[subsection]{Lemma}

\newtheorem{defn}[subsection]{Definition}

\theoremstyle{definition}

\theoremstyle{remark}
\newtheorem{rem}[subsection]{Remark}

\newtheorem{exam}[subsection]{Example}

\newtheorem{nota}[subsection]{Notation}

\newcommand{\no}{\noindent}

\def\vol{{\rm vol}}

\newcommand{\La}{\Lambda}
\newcommand{\la}{\lambda}
\newcommand{\al}{\alpha}
\newcommand{\C}{\mathbb C}
\newcommand{\Q}{\mathbb Q}
\newcommand{\Z}{\mathbb Z}
\newcommand{\R}{\mathbb R}

\def\cX{{\mathcal X}}
\def\cA{{\mathcal A}}
\def\cB{{\mathcal B}}
\def\cC{{\mathcal C}}
\def\cD{{\mathcal D}}

\def\cH{{\mathcal H}}
\def\cO{{\mathcal O}}
\def\cK{{\mathcal K}}
\def\cM{{\mathcal M}}
\def\cL{{\mathcal L}}

\def\cH{{\mathcal H}}
\def\cR{{\mathcal R}}

\def\cT{{\mathcal T}}
\def\cY{{\mathcal Y}}

\def\sL{{\mathsf L}}
\def\sC{{\mathsf C}}
\def\sc{{\mathsf c}}
\def\sd{{\mathsf d}}
\def\sn{{\mathsf n}}
\def\sI{{\mathsf I}}
\def\sJ{{\mathsf J}}

\def\ga{{\mathbf G}_a}

\def\dv{{\rm div}}

\def\Pf{{\rm Pf}}
\def\Gr{{\rm Gr}}

\def\Aut{{\rm Aut}}
\newcommand{\G}{\mathrm G}

\newcommand{\Ker}{{\rm Ker}}

\newcommand{\codim}{{\rm codim}}

\def\k{{\mathbf k}}

\newcommand{\PGL}{{\rm PGL}} 
\def\Val{{\rm Val}}
\def\Spect{{\rm Spect}}
\def\Spec{{\rm Spec}}
\def\Ad{{\rm Ad}}
\def\pr{{\it pr}}

\def\Id{{\rm Id}}
\def\GR{{\rm Gr}}

\def\ra{\rightarrow}

\def\A{{\mathbb A}}
\def\C{{\mathbb C}}

\def\P{{\mathbb P}}
\def\Q{{\mathbb Q}}
\def\L{{\mathrm L}}

\def\Z{{\mathbb Z}}
\def\C{{\mathbb C}}
\def\N{{\mathbb N}}

\def\bG{{\mathrm G}}
\def\bN{{\mathrm N}}
\def\bM{{\mathrm M}}
\def\bH{{\mathrm H}}
\def\bH{{\mathrm H}}

\def\bfK{{\mathrm K}}

\def\me{{\mathfrak e}}
\def\mg{{\mathfrak g}}
\def\mo{{\mathfrak o}}

\def\mn{{\mathfrak n}}
\def\mh{{\mathfrak h}}
\def\mk{{\mathfrak k}}
\def\mm{{\mathfrak m}}
\def\ml{{\mathfrak l}}
\def\mr{{\mathfrak r}}

\def\mS{{\mathfrak S}}
\def\mT{{\mathfrak T}}
\def\mz{{\mathfrak z}}
\def\mZ{{\mathfrak Z}}
\def\mU{{\mathfrak U}}

\def\sym{{\rm sym}}
\def\Pic{{\rm Pic}}

\def\GL{{\rm GL}}
\def\PGL{{\rm PGL}}

\def\Ind{{\rm Ind}}
\def\rk{{\rm rk\,}}

\def\zZ{{\mathcal Z}}
\def\eps{{\epsilon}}
\def\tr{{\rm tr}}

\def\rB{{\rm B}}
\def\rb{{\rm b}}

\def\rG{{\rm G}}

\def\rJ{{\rm J}}
\def\rP{{\rm P}}
\def\tD{{\mathcal D}}
\def\dD{{\Delta}}
\def\End{{\rm End}}
\def\ovl{\overline}
\def\ba{\backslash}

\begin{document}

\author{Joseph Shalika}
\author{Yuri Tschinkel}
\address{Courant Institute of Mathematical Sciences, N.Y.U. \\
 251 Mercer str. \\
 New York, NY 10012, USA}
\email{tschinkel@cims.nyu.edu}

\address{Simons Foundation\\
160 Fifth Avenue\\
New York, NY 10010\\
USA}

\title[Height zeta functions]{Height zeta functions of\\
equivariant compactifications of \\
unipotent groups}

\maketitle

\begin{abstract}
We prove Manin's conjecture for 
bi-equivariant compactifications of unipotent groups.
\end{abstract}

\tableofcontents

\setcounter{section}{0}

\section*{Introduction}
\label{sect:introduction}

Let $F$ be a number field and $X$ an algebraic variety over $F$; we write
$X(F)$ for the set of its $F$-rational points. 
The {\em height} of an $F$-rational
point  ${\bf x}=(x_0:...:x_n)\in \P^n(F)$ of a projective space is
given by
$$
H({\bf x}):=\prod_v \max_{j}|x_j|_v,
$$
where the product is over the set of all
valuations of $F$ and $|\cdot |_v$ is the standard $v$-adic absolute value. 
Let $\bG$ be a linear algebraic group over $F$ and 
$$
\rho\,:\,\bG\ra \PGL_{n+1}
$$
a projective rational representation of $\bG$.   
Assume that there exists a point $e\in \P^n$ with trivial
stabilizer (under the action of $\rho(\bG)$). 
We are interested in the asymptotics of 
$$
N(B):=\{ \gamma\in \bG(F)\,|\, H(\rho(\gamma)\cdot e)\le B \}, \quad   B\ra \infty.
$$
An alternative geometric description of this problem is as follows:
Consider the Zariski closure $X\subset \P^n$ 
of the orbit 
$$
\{\rho(\gamma)\cdot e\,|\, \gamma\in \bG(F)\}.
$$
Then $X$ is an {\em equivariant} compactification of 
$\bG$, embedded by a $\bG$-linearized (ample) line bundle $L$. 
Choosing a particular height
in the ambient projective space 
amounts to choosing an adelic 
{\em metrization} $\cL:=(L,\|\cdot\|_v)$ of $L$ 
(see Section~\ref{sect:metrics} for the definitions).   
In this setup, the problem is to understand 
\begin{equation}
\label{eqn:n}
N(\cL,B):=
\{ \gamma\in \bG(F)\,|\, H_{\cL}(\gamma)\le B\}, \quad B\ra \infty,
\end{equation}
where $H_{\cL}$ is the height defined by $\cL$. 

In this paper we consider smooth projective
{\em bi-equivariant} compactifications of a 
unipotent group $\bG$ over $F$.
This means that $\bG$ is contained in 
$X$ as a Zariski open subset and 
that the natural left {\em and} right actions of $\bG$ on itself 
extend to left and right actions of $\bG$ on $X$.
Alternatively, one may think of $X$ as an equivariant 
compactification of the homogeneous space $\bG\times \bG/\bG$. 

The main result is the determination of 
the asymptotic (\ref{eqn:n}) for arbitrary 
bi-equivariant compactifications $X$ as above and $\mathcal L=-\mathcal K_X$, 
the {\em anticanonical} 
line bundle equipped with a smooth adelic metrization, proving 
Manin's conjecture \cite{FMT} and its refinement by Peyre~\cite{peyre} for this class of varieties. 
This generalizes the theorem for equivariant 
compactifications of the Heisenberg group proved in \cite{shalika-t-heis}.

It turns out that the geometric language
is more adequate for the description of the asymptotic
behavior. More precisely, denote by $\Pic(X)$ the Picard group of 
$X$, this is a free abelian group generated by the classes of the irreducible boundary
components $D_{\alpha}$, $\alpha\in \mathcal A$ 
(we will generally identify divisors and their classes in $\Pic(X)$). 
Our main result is a proof of Manin's conjecture:

\begin{Theo}
\label{thm:main}
Let $X$ be a smooth projective bi-equivariant 
compactification of $\bG$, with boundary
$$
X\setminus \bG = \cup_{\alpha\in \mathcal A}D_{\alpha}
$$
a normal crossings divisor consisting of geometrically irreducible components.   
Then 
$$
N(-\mathcal K_X, B)=\frac{\tau(-\mathcal K_X)}{(b-1)!} 
B \log(B)^{b-1}(1+o(1)), \quad \text{ as } B\ra \infty,
$$   
where $b=\rk\Pic(X)=\#\mathcal A$ is the number of boundary components
and $\tau(-\cK_X)$ is the Tamagawa number defined by 
Peyre \cite{peyre}.
\end{Theo}

We now give an outline of the proof. 
In Section~\ref{sect:nill} 
we recall some basic structural results concerning nilpotent
algebras and unipotent groups. 
In Section~\ref{sect:coad} we discuss coadjoint orbits and 
their parametrization and in Section~\ref{sect:integral} integral structures. 
In Section~\ref{sect:reprs} we 
collect facts regarding unitary representations of unipotent groups
over the adeles. In Section~\ref{sect:ug} we study the
action of the  universal enveloping algebra in representation spaces.
All of the above material is standard and can be found in the books
\cite{cg}, \cite{dixmier} and the papers \cite{kirillov0}, \cite{moore}. 

\

In Section~\ref{sect:geom} 
we turn to equivariant compactifications of unipotent groups and 
describe the relevant geometric invariants and 
constructions. In Section~\ref{sect:metrics} 
we introduce the height pairing
$$
H=\prod_v H_v\,:\, \Pic(X)_{\C}\times \bG(\A)\ra \C,
$$
generalizing the usual heights, and 
the height zeta function
\begin{equation}
\label{eqn:zz}
\zZ({\bf s};g):=\sum_{\gamma\in \bG(F)}H({\bf s};\gamma g)^{-1}.
\end{equation}
By the projectivity of $X$, the series 
converges to a function which is continuous and bounded in $g$ and
holomorphic in ${\bf s}$ for $\Re({\bf s})$ 
contained in some cone $\Lambda\subset \Pic(X)_{\R}$.
Our goal is establish its analytic properties, and in particular, to obtain a meromorphic
continuation of the 1-parameter {\em height zeta function}
$$
\zZ(sL)=
\sum_{\gamma\in \bG(F)} H_{\mathcal L}(sL,\gamma)^{-1},
$$
the restriction of the
multiparameter zeta function $\zZ({\bf s};g)$ to
the complex line through $L$ and the identity $g=e\in G(\A_F)$. 

To describe the polar set, we use the classes $D_{\alpha}$ as a basis of 
$\Pic(X)$. In this basis, the pseudo-effective 
cone $\La_{\rm eff}(X)\subset \Pic(X)_{\R}$
consists of classes $(l_{\al})\in \Pic(X)_{\R}$ 
with $l_{\al}\ge 0$ for all $\al$. Let 
$$
-K_X=\kappa=\sum_{\alpha\in \mathcal A}\kappa_{\alpha}D_{\alpha} \in \Pic(X)_{\R},
$$
be the anticanonical class. We know (see Proposition~\ref{prop:geometry}) that 
$\kappa_{\alpha}\ge 2$, for all $\alpha\in \mathcal A$. 
Conjecturally, analytic properties of {\em height zeta functions}
$\zZ(sL)$ depend on the location of $L=(l_{\al})\in \Pic(X)$
with respect to the anticanonical class and the cone
$\Lambda_{\rm eff}(X)$
(see \cite{FMT}, \cite{peyre} and \cite{BT}).
Precisely, define
\begin{itemize}
\item 
$a(L):=\inf\{ a\,\, | \,\, aL+K_X\in \La_{\rm eff}(X)\}= 
\max_{\al}(\kappa_{\al}/l_{\al});$
\item $b(L):=\#\{ \al\,|\, \kappa_{\al}=a(L)l_{\al}\}$;
\item $\cC(L):=\{ \al \,|\, \kappa_{\al}\neq a(L)l_{\al}\}$;
\item $c(L):=\prod_{\al \notin \cC(L)} l_{\al}^{-1}$.
\end{itemize}
Then, conjecturally,      
\begin{equation}
\label{eqn:zeta-want}
\zZ(sL)=\frac{c(L)\tau(\cL)}{(s-a(L))^{b(L)}} +
\frac{h(s)}{(s-a(L))^{b(L)-1}},
\end{equation}
where $h(s)$ is a holomorphic function 
(for $\Re(s)>a(L)-\delta$, some $\delta>0$) and $\tau(\cL)$ 
is a positive real number. Given this, 
Tauberian theorems imply
$$
N(\cL,B)=\frac{c(L)\tau(\cL)}{a(L)(b(L)-1)!} 
B^{a(L)} \log(B)^{b(L)-1}(1+o(1)),
$$
as $B\ra \infty$, for certain constants $\tau(\mathcal L)$ 
defined in \cite{BT}.   
Here we establish this for $L=-K_X$, 
via a spectral expansion of $\zZ({\bf s};g)$ from Equation~\eqref{eqn:zz}. 

\

The bi-equivariance of $X$ implies that $H$ is 
invariant under the action {\em on both sides} 
of a compact open subgroup $\bfK$ of the
finite adeles $\bG(\A_{\rm fin})$. 
Furthermore, we assume that  
$H_v$ is smooth for archimedean $v$. 
We observe that
$$
\zZ({\bf s};g)\in \sL^2(\bG(F)\backslash \bG(\A))^{\bfK}
$$
and we proceed to analyze its spectral decomposition.
We get a formal identity
\begin{equation}
\label{eqn:formalz}
\zZ({\bf s};g)=\sum_{\varrho} \zZ_{\varrho}({\bf s};g),
\end{equation}
where the summation is over all irreducible unitary representations
$(\varrho,\cH_{\varrho})$ of $\bG(\A)$ 
occurring in the right regular representation of $\bG(\A)$ in  
$\sL^2(\bG(F)\ba \bG(\A))$. 
These  are parametrized by
$F$-rational orbits $\cO=\cO_{\varrho}$ under the coadjoint 
action of $\bG$ on the dual of 
its Lie algebra $\mg^*$.    
The relevant orbits are {\em integral} - 
there exists a lattice in $\mg^*(F)$ 
such that $\zZ_{\varrho}({\bf s};g)=0$ unless 
the intersection of $\cO$ with this lattice is nonempty.
The pole of highest order 
is contributed by the trivial representation 
and integrality insures that this representation
is ``isolated''.

\

Let $\varrho$ be a representation as above. 
Then $\varrho$ arises from some 
$$
\pi = \Ind_{\bM}^{\bG}(\psi),
$$
where  $\bM\subset \bG$ is an $F$-rational 
subgroup and $\psi$ is a certain character of $\bM(\A)$ (see Proposition~\ref{prop:decomp}).
In particular, for the trivial representation, 
$\bM=\bG$ and $\psi$ is the trivial character.
Further, there exists a finite set of  
places $S=S_{\varrho}$ such that
$\dim\, \varrho_{v}=1$  for $v\notin S$ 
and consequently 
\begin{equation}
\label{eqn:zeta}
\zZ_{\varrho}({\bf s};g')
= \zZ^{S}({\bf s};g') \cdot \zZ_{S}({\bf s};g'),
\end{equation}
where
$$
\zZ^{S}({\bf s};g'):=
\prod_{v\notin S} 
\int_{\bM(F_v)}
H_v({\bf s};m_vg_v')^{-1}\ovl{\psi}(m_vg_v')dm_v,
$$
(with an appropriately normalized Haar measure 
$dm_v$ on $\bM(F_v)$) and  
the function $\zZ_{S}$ is the projection of
$\zZ$ to $\otimes_{v\in S}  \, \varrho_{v}$.

\

The first key result is the explicit computation of 
{\em height integrals}:
$$
\int_{\bM(F_v)}
H_v({\bf s};m_vg_v')^{-1}\ovl{\psi}(m_vg_v')dm_v
$$
for almost all $v$ (see Section~\ref{sect:prop}). 
This has been done in \cite{CLT} for
equivariant compactifications of additive groups $\ga^n$; 
the same approach works here as well. We regard the height integrals
as geometric versions of Igusa's integrals (see \cite{CLT-igusa}).

\

\no
For the trivial representation and $v\notin S$, 
we have
\begin{equation}
\label{eqn:hi}
\int_{\bG(F_v)} H({\bf s};g_v)^{-1}dg_v =
q_v^{-\dim X}\left( \sum_{A\subseteq \cA} D^0_{A}(\k_v)  
\prod_{\al\in A} \frac{q_v-1}{q_v^{s_{\al}-\kappa_{\al}+1}-1} \right),
\end{equation}
where 
$$
D_A:= \cap_{\al\in A}D_{\al},\,\,\, \, 
D_A^0:=D_A\setminus \cup_{A'\supsetneq A}D_{A'} 
$$
and $q_v$ is the cardinality of the residue field $\k_v$ 
at $v$. Restricting to the line through $-K_X$, we find that the  
resulting Euler product $\zZ^{S}(-sK_X)$ is regularized by a product of (truncated)
Dedekind zeta functions, thus is holomorphic for $\Re(s)>1$, 
admits a meromorphic continuation to $\Re(s)>1-\delta$, for some $\delta>0$, and  
has an isolated pole of order $\rk\, \Pic(X)$
at $s=1$, with the expected leading coefficient $\tau(-\mathcal K_X)$.
Similarly, we identify the poles 
of $\zZ^{S}$ for nontrivial representations: again, they are regularized by products of (truncated)
Dedekind zeta functions and thus admit 
a meromorphic continuation to the same
domain, with at most an isolated pole at $s=1$; but the order of the pole 
at $s=1$ is strictly smaller than $\rk\, \Pic(X)$.

\

Next we need to estimate $\dim\, \varrho_{v}$ 
and the local integrals 
for nonarchimedean $v\in S$ (see Sections~\ref{sect:mult}
and \ref{sect:prop}). Then we turn to archimedean places.
Using integration by parts, we prove in Lemma~\ref{lemm:extim} 
that for all $\eps>0$ and all (left or right) $\bG$-invariant differential operators 
$\partial$ there exist constants
$\sc = \sc(\eps,\partial)$ and $N=N(\partial)$ such that
\begin{equation}
\label{eqn:c}
\int_{\bG(F_v)} |\partial H_{v}({\bf s}; g_v)^{-1} |_v\, dg_v\le 
\sc \cdot \|{\bf s}\|^N,
\end{equation}
for all ${\bf s}$ with $\Re(s_{\alpha})>\kappa_{\alpha}-1 +\eps$, for all $\alpha\in \mathcal A$. 

\

Let $v$ be real. It is known that $\varrho_{v}$ admits a standard model
$(\pi_{v},\sL^2(\R^r))$, where $2r=\dim\, \cO$.
More precisely, there exists an isometry
$$
j \,\,:\,\, (\pi_{v},\sL^2(\R^r)) \ra 
(\varrho_{v},\cH_{v}),
$$
an analog of the $\Theta$-distribution.
Moreover, 
the universal enveloping algebra $\mU(\mg)$ surjects onto
the Weyl algebra of differential operators with
polynomial coefficients acting on the smooth vectors 
$\sC^{\infty}(\R^r)\subset \sL^2(\R^r)$. 
In particular, we can find an operator $\Delta$ acting  
as the $r$-dimensional harmonic oscillator
$$
\prod_{j=1}^r(\frac{\partial^2}{\partial x_j^2} -a_jx_j^2),
$$ 
with $a_j>0$. 
We choose an orthonormal basis of $\sL^2(\R^r)$ consisting 
of $\Delta$-eigenfunctions $\{ \tilde{\omega}_{\la}\}$
(which are well known) and analyze
$$
\int_{\bG(F_v)} H_v({\bf s};g_v)^{-1}
\ovl{\omega}_{\la}(g_v)dg_v,
$$
where $\omega_{\la}=j(\tilde{\omega}_{\la})$. 
Using integration by parts and (\ref{eqn:c}) we find that for all $n\in \N$ 
there exist constants
$\sc=\sc(n,\Delta)$ and  $N\in \N$ such that 
this integral is bounded by
\begin{equation}
\label{eqn:bound}
\sc \cdot \la^{-n}\cdot \|{\bf s}\|^{N},
\end{equation}
for ${\bf s}$  with $\Re(s_{\alpha})>\kappa_{\alpha}-1+\eps$, for all $\alpha$. 
This estimate suffices to conclude that for {\em each} $\varrho$
the function $\zZ_{S_{\varrho}}$ is holomorphic in a neighborhood of $\kappa$; indeed 
it will be majorized by 
$$
\sum_{\lambda} \lambda^{-n}, 
$$
the spectral zeta function of a compact manifold, 
which converges for sufficiently large $n\ge 0$ (see Section~\ref{sect:prop} and the Appendix).

\

Now the issue is to prove the 
convergence of the sum in (\ref{eqn:formalz}). 
Using any element $\partial\in \mU(\mg)$ acting 
in $\cH_{\varrho}$ by a scalar $\la(\partial)\neq 0$ 
(for example, any element in the center of $\mU(\mg)$) 
we can improve the bound (\ref{eqn:bound}) to
$$\
 \sc \cdot \la^{-n_1} \cdot \la(\partial)^{-n_2} \cdot \|{\bf s}\|^N
$$
(for any $n_1,n_2\in \N$ and some constants $\sc=\sc(n_1,n_2,\Delta,\partial)$ and 
$N=N(\Delta,\partial)$. 
However, we have to insure the 
uniformity of such estimates over the set of all 
$\varrho$. This relies on a parametrization of
coadjoint orbits. There is a finite set 
$\Sigma$ of ``packets''
of coadjoint orbits, each parametrized 
by a locally closed subvariety $Z_\sigma\subset \mg^*$, 
and for each $\sigma$ a finite set 
of $F$-rational polynomials $\{ P_{\sigma,j}\}$ 
on $\mg^*$ such that each $P_{\sigma,j}$ is
invariant under the coadjoint action and nonvanishing on the stratum $Z_\sigma$.
Consequenty, the corresponding derivatives 
$$
\partial_{\sigma,j}\in \mU(\mg)
$$
act in $\cH_{\varrho}$ by multiplication by the scalar
$$
\lambda_{\varrho,j}(\ell)=P_{\sigma,j}(2\pi i \ell), \, \quad \ell\in \mathcal O.
$$
Recall that $\ell$ varies over a lattice; applying several times
$\partial_{\sigma}=\prod_j\partial_{\sigma,j}$
we obtain the uniform convergence of the right hand side 
in (\ref{eqn:formalz}).

\

The last technical point is to prove that both 
expressions (\ref{eqn:zz}) and (\ref{eqn:formalz}) 
for $\zZ(-sK_X;g)$ define {\em continuous} functions on 
$\bG(F)\ba \bG(\A)$. Then (\ref{eqn:formalz})
gives the desired meromorphic continuation of $\zZ(-sK_X;e)$.

\

The techniques described above should allow the treatment of arbitrary height functions; here we  
restricted to the anticanonical height $H_{-\mathcal K_X}$ as in the original conjecture of Manin \cite{FMT}, 
to avoid some technical issues with $L$-primitive fibrations (see \cite{BT} and \cite{CLT}).

\

{\bf Acknowledgements.}
The second author was partially supported by NSF grants 0739380, 
0901777, and 1160859. He is very grateful to the referees for comments and suggestions that helped to 
improve the exposition.

\

\section{Nilpotent Lie algebras and unipotent groups}
\label{sect:nill}

In this section we recall basic properties of 
nilpotent Lie algebras and unipotent groups.                 
We work over a field $F$ of characteristic zero.

\subsection{Nilpotent algebras}
\label{sect:na}

Let $\mg=(\mg,[,])$ 
be an $n$-dimensional Lie algebra over $F$: an affine space
over $F$ of dimension $n$ together with a bracket 
$[\cdot ,\cdot]$ satisfying the Jacobi identity. 
Denote by $\mz_{\mg}$ the center of $\mg$.
For a subset $\mh\subset \mg$ we denote by
$$
\mn_{\mg}(\mh):=\{ X\in \mg\,|\, [X,\mh]\subset \mh\} 
$$
its normalizer and by
$$
\mz_{\mg}(\mh):=\{ X\in \mg\,|\, [X,Y]=0, \forall Y\in \mh\}
$$
its centralizer. 
Let 
$$
\mg_1\subset\mg_2\subset ...\subset \mg_k\subset \mg
$$
be a sequence of subalgebras. 
A {\em weak} Malcev basis 
through this sequence is a basis 
$(X_1,...,X_n)$ of $\mg$ 
such that
\begin{itemize}
\item for all $j\in 1,...,k$ there exists an $n_j$
such that $\mg_j=\langle X_1,..., X_{n_j}\rangle$;
\item for all $i=1,...,n$ the $F$-vector space
$\langle X_1,...,X_i\rangle$ is a Lie subalgebra.
\end{itemize}
Assume that all $\mg_j$ above are ideals. A {\em strong}
Malcev basis through this sequence is a 
weak Malcev basis such that
\begin{itemize}
\item for all $i=1,...,n$ the $F$-vector space
$\langle X_1,...,X_i\rangle$ is an ideal.
\end{itemize}

\noindent
The {\em ascending central series} of $\mg$ is defined as
$$
\begin{array}{ccl}
\mg_0 & := & 0;\\
\mg_j & := & \{ x\in \mg\, |\, [x,\mg]\subseteq \mg_{j-1}\}.
\end{array}
$$
From now on we will assume that
$\mg$ is {\em nilpotent}, that is, there exists an $n$ 
such that $\mg_n=\mg$. 

\begin{exam}
\label{exam:exa}
Some common examples are: 
\begin{itemize}
\item the Heisenberg algebra 
$\mh_3:=\langle X,Y,Z\rangle$, 
$[X,Y]=Z$;
\item the upper-triangular algebra $\mn_n\subset \mg\ml_n$;
\item the algebra $\mk_4=\langle X_1,X_2,X_3,Y\rangle$:
$[X_i,X_j]=0, [Y,X_i]=X_{i-1}$.  
\end{itemize}
\end{exam}

\begin{lemm}
\label{lemm:passing}
If $\mg$ is nilpotent then for any ascending sequence
of algebras (resp. ideals) there exists a weak (resp. strong)
Malcev basis passing through it. 
\end{lemm}

\begin{proof}
Indeed, for any subalgebra
$$
\mh\subsetneq\mn_{\mg}(\mh),
$$
and for any $X\in \mn_{\mg}(\mh)\setminus \mh$
the vector space $\mh\oplus FX$ is a subalgebra. 
Same argument works for ideals. 
\end{proof}

There is no canonical choice of a 
Malcev basis through a given subalgebra. 

\begin{lemm}(Kirillov's lemma)
\label{lemm:kir}
Let $\mg$ be a noncommutative nilpotent Lie algebra with 
1-dimensional center $\mz_{\mg}(\mg)=\langle Z\rangle $. Then 
there exist $X,Y\in \mg$ such that 
\begin{itemize}
\item $[X,Y]=Z$;
\item $\mg = \mz_{\mg}(Y)\oplus FX$. 
\end{itemize}
\end{lemm}

\begin{proof}
Choose some $Y\in \mg_2\setminus \mg_1$. Then 
$\mg_0:=\mz_{\mg}(Y)$ is a subalgebra of codimension one and there 
is an $X$ in its complement as required. 
\end{proof}

\begin{nota}
\label{nota:redq}
We refer to the quadruple $(Z,Y,X, \mg_0)$ in 
Lemma~\ref{lemm:kir} as a reducing quadruple. 
\end{nota}

\subsection{Polarizations}
\label{sect:pol}

Denote by $\mg^*$ the dual
Lie algebra. Each $\ell \in \mg^*$ determines
a skew-symmetric bilinear form 
$$
\begin{array}{ccccc}
\rB_{\ell} & : &  \mg\times \mg & \ra &  F\\
           &   &    (X , Y)     & \ra & \ell([X,Y]).
\end{array}
$$ 
For any subalgebra $\mh\subset\mg$ denote by 
$$
\mr_{\ell}(\mh):=\mh\cap \mh^{\perp_{\ell}}
=\{ h\in \mh\,|\, \ell([h,h'])=0,\, \,\forall \, h'\in \mh\}
$$ 
its {\em radical} with respect to $\rB_{\ell}$. Clearly, 
the maximum dimension of an isotropic subspace in $\mathfrak g$ is
$$
d=\dim \mathfrak r_{\ell} +\frac{1}{2}(\dim \mathfrak h - \dim \mathfrak r_{\ell} ).
$$

\begin{defn}
\label{defn:pola}
A subalgebra $\mm_{\ell}\subset \mg$
is called {\em polarizing} for $\ell$ if
\begin{itemize}
\item $\mm_{\ell}$ 
is isotropic for $\rB_{\ell}$, that is, $\rB_{\ell}(m,m')=0$
for all $m,m'\in \mm_{\ell}$;
\item $\dim \mm_{\ell} $ is the maximal possible dimension $d$ for isotropic subspaces. Such subalgebras exist, and all 
have the same dimension. 
\end{itemize}
\end{defn}

\begin{exam}
For the Heisenberg algebra $\mh_3$ and any $\ell$ with 
$\ell(Z)\neq 0$ a polarizing subalgebra is the ideal
$\mm_{\ell}=\langle Z,Y\rangle$. 
\end{exam}

\begin{rem}
\label{rem:many}
A polarizing algebra $\mm_{\ell}$ is not necessarily an ideal.
An $\ell\in \mg^*$ can have many polarizing 
subalgebras. In general, 
there does not exist a finite set of subalgebras 
such that for each $\ell\in \mg^*$ one 
of the subalgebras in this set is polarizing for $\ell$.
\end{rem}

A {\em canonical} construction of a polarizing 
algebra (by Vergne~\cite{vergne}) 
goes as follows: fix a strong Malcev basis
$(X_1,...,X_n)$ for $\mg$. 
Put
$$
\mm_{\ell}:=\sum_{j=1}^n \mr_{\ell}(\mg_j),
$$
where $\mg_j:=\langle X_1,...,X_j\rangle$ and 
$\mr_{\ell}(\mg_j)$ is the radical of $\mg_j$ with respect to $\rB_\ell$. 

\

\noindent
Alternatively, a polarizing subalgebra may be constructed
inductively: 

\

{\bf Case 1.} If 
$$
\mz_{\ell}:=\mz_{\mg}\cap \Ker(\ell)\neq 0
$$
consider the projection
$$
\pr\,:\, \mg\ra \mg_0:=\mg/\mz_{\ell}
$$ 
and write $\ell_0$ for 
the induced linear form on $\mg_0$. 
If $\mm_{\ell_0}\subset \mg_0$ 
is a polarizing algebra for
$\ell_0$ the preimage $\pr^{-1}(\mm_{\ell_0})$
is a polarizing algebra for $\ell$.   

\

{\bf Case 2.} Otherwise, $\mz(\mg)=\langle Z\rangle $ 
and $\ell(Z)\neq 0$. 
Then there exists a $Y\in \mg_2\setminus \mg_1$ 
such that $\codim\, \mz_{\mg}(Y)=1$ (by Lemma~\ref{lemm:kir}). 
Let $\ell_Y$ be the restriction of $\ell$ to $\mz_{\mg}(Y)$
and $\mm_Y$ a polarizing algebra for $\ell_Y$ in $\mz_{\mg}(Y)$. 
Then $\mm_{\ell}=\mm_Y$.

\begin{prop}
\label{prop:ml-rat}
Let $Z\subset \mg^*$ be an algebraic variety, defined over $F$.
There exists a Zariski open subset $Z^0\subset Z$, a positive integer  
$k\le \dim \mg$ and an $F$-morphism 
$$
{\rm pol}\,:\, Z^0\ra \Gr(k,\mg)
$$ 
such that for every point $\ell$ in $Z^0$ the image ${\rm pol}(\ell)$
in the Grassmannian of $k$-dimensional subspaces in $\mg$ 
corresponds to a polarizing subalgebra for $\ell$.
\end{prop}

\begin{proof}
Consider $\mg^*$ over the function field of $Z$ and apply Vergne's 
construction to the generic point. 

Alternatively, consider the subvariety of all subalgebras $\mm\subset \mg$ 
over the function field $F(Z)$ of dimension $k$ such that $\ell([m,m])=0$, with 
$\ell\in \mg^*(F(Z))$.
Take the maximal $k$ such that this variety has an $F(Z)$-rational point. 
This point defines an $F(Z)$-rational point in $\Gr(k,\mg)$. Specializing, 
we get polarizations on some open subset $Z^0\subset Z$. 
\end{proof}

\subsection{Unipotent groups}
\label{sect:uF}

Let $V$ be a finite dimensional vector space over
$F$ and $\bN\subset \GL(V)$ 
the subgroup of all upper-triangular unipotent matrices. 
Denote by $\mn$ the $F$-vector space 
of all upper-triangular nilpotent  
matrices. 
The (standard) maps
$$
\begin {array}{cccc}
\exp: & \mn & \ra &  \bN \\
\log: & \bN & \ra &  \mn
\end{array}
$$
are biregular $F$-morphisms (polynomial maps) 
between algebraic varieties. 

\

Let $\bG$ be a (connected) unipotent linear 
algebraic group defined over $F$. 
Then there exists an $F$-rational representation 
$$
\rho_F\,:\, \bG\ra \GL(V),
$$
for some $V$, realizing $\bG$ as a closed subgroup of $\bN$. 
We fix this representation. 
Then 
$$
\mg:=\log(\bG)\subset \mn
$$ 
is the Lie algebra of $\bG$.
This coincides with the usual definition
of $\mg$ as the $F$-algebra of left-invariant 
$F$-derivations of the algebra or rational
functions $F[G]$.

\section{Coadjoint orbits}
\label{sect:coad}

\subsection{Orbits}

Both $\mg$ and its dual $\mg^*$ are defined over $F$.
For all fields $E/F$ we can consider the $E$-rational points
of $\mg$ and $\mg^*$, which we denote by $\mg(E)$,
resp. $\mg^*(E)$. 

\

Denote by $\Ad$ (resp. $\Ad^*$) the adjoint
(resp. coadjoint) action of $\bG$ on $\mg$ (resp. $\mg^*$), 
both are algebraic actions defined over $F$.
Let $\cO_{\ell}$ be the coadjoint orbit through $\ell\in \mg^*(F)$.
It is a {\em symplectic} algebraic 
variety: the skew-symmetric bilinear form 
$$
\begin{array}{ccccc}
\rB_{\ell}& : &  \mg\times \mg&\ra &  F\\
          &   &   (X,Y)       & \mapsto & \ell([X,Y])
\end{array}
$$
descends to a {\em nondegenerate} algebraic 2-form 
$\Omega_{\ell}$ on the orbit $\cO_{\ell}$.

\begin{lemm}
\label{lemm:tangent}
The map 
$$\begin{array}{ccl}
\bG & \ra &  \cO_{\ell}\\
g   & \mapsto & \Ad^*(g)^{-1}\circ \ell
\end{array}
$$ 
induces an exact sequence
$$
0\ra \mr_{\ell}\ra \mg\stackrel{\pr_{\ell}}{\longrightarrow} 
\cT_{\ell}(\cO_{\ell})\ra 0,
$$
where $\cT_{\ell}$ is the tangent space at $\ell$,
$\mr_{\ell}$ is the radical of the skew-symmetric form $\rB_{\ell}$, and 
$\pr_{\ell}$ is the map to the tangent space of the orbit. 
\end{lemm}

\begin{lemm}
\label{lemm:restr}
Let $\mg$ be a nilpotent Lie algebra with 1-dimensional center
and reducing quadruple $(Z,Y,X,\mg_0)$.
Let  $\ell\in \mg^*$ be such that $\ell(Z)\neq 0$. Let  
$$
\pr\,:\, \mg\ra \mg_0.
$$  
be the projection and $\ell_0$ the restriction of $\ell$
to $\mg_0$. 
Then $\mr_{\ell}$ is properly contained in $\mr_{\ell_0}\subseteq \mg_0$,
$$
\pr(\cO_{\ell})=\sqcup_{t\in F} \Ad^* \exp(tX)(\cO_{\ell_0})
$$ 
and $\pr^{-1}(\mathcal O_{\ell_0})=\mathcal O_\ell$. 
\end{lemm}

\begin{proof}
See \cite{cg}, p. 69.
\end{proof}

\

\subsection{Basic invariant theory}
\label{sect:inva}

In this section we describe the geometric structure of
the set of coadjoint orbits. The main result is the following 

\begin{prop}
\label{prop:dec}
One has a decomposition
$$
\mg^*=\sqcup_{\sigma\in \Sigma} Z_{\sigma}
$$ 
into a finite union of irreducible algebraic varieties 
$\{ Z_{\sigma}\}_{\sigma\in \Sigma}$.
For each $Z_{\sigma}$ there exists a 
finite set of polynomials $\{ P_{\sigma,j}\}_{j\in \sJ_{\sigma}}$ 
on $\mg^*$ separating the orbits: $P_{\sigma,j}$ are invariant on 
each orbit in $Z_{\sigma}$ and for every pair $\ell,\ell'\in \mg^*$ 
contained in different orbits $\cO,\cO'\subset Z_{\sigma}$ 
there exist $j,j'\in \sJ_{\sigma}$ 
such that $P_{\sigma, j}(\ell)\neq P_{\sigma, j'}(\ell')$.
\end{prop}

We follow the exposition in \cite{popov-vinberg}. 
Let $V$ be an irreducible algebraic variety over a field $F$. 
Denote by $F[V]$ the ring of regular functions on $V$ and by
$F(V)$ its function field. Let $\bG$ be an algebraic group over $F$. 
A {\em regular} action of $\bG$ on $V$ is an $F$-homomorphism 
$\rho_{\rm reg}\,:\, \bG\ra \Aut(V)$ such that the induced map 
$\rho\,:\, \bG\times V\ra V$ is a morphism of algebraic varieties. 
A {\em rational} action
is a homomorphism $\rho_{\rm rat}\,:\, \bG\ra {\rm Bir}(V)$ 
such that the induced map $\rho$ is defined and coincides with 
some {\em rational} map $\rho^0$ 
on a dense Zariski open subset. 
An {\em orbit} $\cO_{\ell}$ through a point $\ell\in V$ is the image in $V$ of
$\bG\times \ell$ under $\rho$.

\

The action $\rho_{\rm rat}$ induces an action on $F(V)$. 
Regular functions $\phi\in F[V]$  
satisfying $g\cdot \phi =\phi$ for all $g\in \bG$ are called
{\em integral} invariants for the action, rational functions $\phi\in F(V)$ 
with the same property 
are called {\em rational invariants}. Integral invariants
form a subalgebra in $F[V]$, denoted by $F[V]^{\bG}$, and rational invariants
a field, $F(V)^{\bG}$. A rational 
invariant $\phi$ separates the orbits $\cO,\cO'$ if $\phi$ is defined
in the points of both orbits and if for all $\ell\in \cO, \ell'\in \cO'$
one has $\phi(\ell)\neq \phi(\ell')$.   
A set $\Phi=\{ \phi\} $ 
of rational invariants separates generic orbits if
there exists a Zariski open dense subset $V^0\subset V$ with the property
that for every pair of points $\ell,\ell'\in V^0$ contained in 
different orbits $\cO,\cO'$ there exists a 
rational invariant $\phi\in \Phi$
separating $\cO,\cO'$. 
In this case $\Phi$ generates the field of rational invariants $F(V)^{\bG}$.

\begin{thm}(Rosenlicht,~\cite{ros})
\label{thm:ros}
For every (rational) action of an algebraic group $\bG$ on an irreducible
algebraic variety there exists a finite set $\Phi$  of rational invariants
separating generic orbits. 
\end{thm}

\begin{thm}
\label{thm:quot}
Assume that  $V$ is an affine algebraic variety and that $\bG$ is a unipotent
group acting on $V$. 
Then every rational invariant is representable as a quotient of  
integral invariants. In particular, there exists a finite set of integral
invariants separating generic orbits.  
\end{thm}

Proposition~\ref{prop:dec} follows: in $V=\mg^*$ 
we find a dense Zariski open subset $V^0$ and a set 
of integral invariants $\{P_{0,j}\}$ separating generic orbits. We can 
stratify the complement $V\setminus V^0$ into a finite disjoint union of 
irreducible affine algebraic varieties of smaller dimension and continue
by induction.

\subsection{Parametrization}
\label{sect:para}

In this section we make the parametrization of coadjoint orbits
more explicit. This will be useful in Section~\ref{sect:mult}, where
we estimate certain multiplicities in terms of relative Pfaffians.

\

We follow the exposition in \cite{cg}, Section 3.
Consider the coadjoint action 
of $\bG$ on the affine space  $V=\mg^*$. 
Fix a strong Malcev basis 
$(X_1,...,X_n)$ for $\mg$, passing through
the ideals $\mg_j$ of the ascending 
central series (see Section~\ref{sect:nill}). 
The dual basis $(\ell_1,...,\ell_n)$
is a Jordan-H\"older basis of $V$, for the
coadjoint action of $\bG$.

\

Denote by 
$$
V_j:= \text{$F$-span of} \,  \{ \ell_{j+1}, ..., \ell_n\} = \mg_j^{\perp},
$$
(the annihilator in $\mg^*$), note that $\mg^*/V_j$ is canonically isomorphic to $(\mg/\mg_j)^*$. 
The canonical projection 
$$
\pr_j\,:\, V\ra V/V_j
$$ 
is $\Ad^*(\bG)$-equivariant.
For ${\bf d}=(d_1,...,d_n)\in \N^n$
consider the subset
$$
Z_{\bf d}:=\{ v\in V\,|\,  
\dim \, \Ad^*(\bG)(\pr_j(v))=d_j, \,\, \forall j\in [1,...,n]\}.
$$
The set ${\bf D}$ of ${\bf d}$ with $Z_{\bf d}\neq \emptyset $ is 
finite and {\em partially} ordered: ${\bf d}\succeq {\bf d}'$ iff
$d_j\ge d_j'$ for all $j\in [1,..., n]$. It has a {\em unique}
maximal element corresponding to ${\bf d}^{\rm max}$ with
$$
d_j^{\max}=\max_{{\bf d}\in {\bf D}}\{ d_j\}
$$ 
for all $j\in [1,...,n]$. Define 
$$
{\bf D}_{1}:=\{ {\bf d}^{\max}\}, \,\,  {\bf D}_{k+1}:=\{ {\bf d}|\, 
{\bf d} \,\,\, {\rm maximal\,\,\, in}\,\,\, {\bf D}\setminus 
(\cup_{k'\le k}{\bf D}_{k'})\}.
$$
Fix an order $>_k$ in each ${\bf D}_k$. This gives 
an order $\succeq $  in ${\bf D}$: 
$$
{\bf D}_{k}\ni {\bf d}\succeq {\bf d}'\in {\bf D}_{k'}
$$ 
if either $k<k'$ 
or (if $k=k'$) ${\bf d} >_k {\bf d}'$. 

\

For ${\bf d}\in {\bf D}$ define
$$
\begin{array}{clccl}
\sI_{\bf d}:= & \{ i\,|\, d_i=d_{i-1}+1\}, &   & 
V^{\sI}_{\bf d}:= & \langle \ell_i\,\,|\,\, i\in \sI_{\bf d}\rangle_F
\\
\sJ_{\bf d}:= & \{ j\,|\, d_j=d_{j-1}\}, &    &  
V^{\sJ}_{\bf d}:=& \langle \ell_j\,\,|\,\, j\in \sJ_{\bf d}\rangle_F,
\end{array}
$$ 
(with $d_0=0$).
We have a decomposition of the $F$-vector space
$$
\mg^*=V^{\sI}_{\bf d}\oplus V^{\sJ}_{\bf d}.
$$

\begin{rem}
A posteriori,  $\sI_{\bf d}$ is always even. 
\end{rem}

\begin{prop}
\label{prop:strata}
The stratification of $V=\mg^*$ into strata $Z_{\bf d}$ satisfies the following properties:
\begin{itemize}
\item For all ${\bf d}\in {\bf D}$ the set 
$\cup_{{\bf d}'\succeq {\bf d}} Z_{\bf d}$
is an $\Ad^*(\bG)$-invariant subset of $V$. 
\item Each $\Ad^*(\bG)$-orbit in the stratum $Z_{\bf d}$ 
meets $V^{\sJ}_{\bf d}$ in exactly one point.
\item $\Sigma_{\bf d}:=Z_{\bf d}\cap V^{\sJ}_{\bf d}$ is 
algebraic (a locally closed subvariety of $\mg^*$).
\item The union $\Sigma=\sqcup_{\bf d} \Sigma_{\bf d}$ 
parametrizes all $\Ad^*(\bG)$-orbits in $\mg^*$.  
\end{itemize}

\noindent
Moreover, 
for each ${\bf d}$ with $\sI_{\bf d}=\{ i_1,...,i_{2k}\}$
there exist rational functions 
$P_1,...,P_n\in F(V\times V^{\sI}_{\bf d}$)
such that: 
\begin{itemize}
\item For each $r=1,...,n$ the restriction of $P_r$ to
$Z_{\bf d}\times V^{\sI}_{\bf d}$ has no poles.
\item For each 
$z\in Z_{\bf d}$ and each $r$ 
the restriction of $P_r$ to $z\times V^{\sI}_{\bf d}$ is
a polynomial and $P_r(z,v)=P_r(z',v)$
for all $z'$ in the orbit $\Ad^*(\bG)(z)$ and all $v\in V^{\sI}_{\bf d}$.
\item The vector $w=\sum_{r=1}^n P_r(z;v)\ell_r$ is the 
unique vector in the $\Ad^*(\bG)$-orbit through 
$z$ whose projection to $V^{\sI}_{\bf d}$ is $v$. 
\end{itemize}
\end{prop}

\begin{proof}
See \cite{cg}, Section 3.1.
\end{proof}

\subsection{Pfaffians}
\label{sect:pfaf}

We fix a strong Malcev basis $(X_1,...,X_n)$ for $\mg$.
Denote by $(\ell_1,...,\ell_n)$ the dual basis of $\mg^*$.
Fix a stratum $Z=Z_{\bf d}$ with   
with $\sI_{\bf d}=\{i_1,...,i_{2k}\}$.
Let $\ell\in Z$ and $\Omega_{\ell}$ be the
canonical symplectic 2-form on $\cO_{\ell}$. Write
$$
\Omega_{\ell}=2\cdot \sum_{i_r<i_{r'}} \rB_{\ell}(X_{i_{r}},X_{i_{r'}})
\ovl{\ell}_{i_r}\wedge \ovl{\ell}_{i_{r'}}
$$
and let $\mu(\ell):=\wedge^{k}\Omega_{\ell}$.
Then 
\begin{equation}
\label{eqn:mu}
\mu(\ell)=2^kk!\Pf(\ell)\cdot \ovl{\ell}_{i_{1}}
\wedge \cdots \wedge \ovl{\ell}_{i_{2k}}
\end{equation}
for some function $\Pf(\ell)$, called the {\em relative Pfaffian}.
Clearly, $\Pf(\ell)$ is a sum of 
terms each of which is a product of factors
of the form 
$$
\rB_{\ell}(X_{i_r},X_{i_{r'}})
$$
and thus a {\em polynomial} function on $\mg^*$. 
Now we notice that the formula~(\ref{eqn:mu}) is well defined for
{\em any} $\ell\in \mg^*$. 

For $\ell\in Z$, we have
$$
\Pf(\ell)^2=\det(\rB_{\ell}(X_{i_r},X_{i_{r'}})).
$$
Since $\rB_{\ell}$ is non-degenerate
on $\cO_{\ell}$, $\Pf(\ell)\neq 0$.
Further, $\Pf(\ell)$ is $\Ad^*(\bG)$-invariant on $Z$.

\

\section{Integral structures}
\label{sect:integral}

We will also need {\em integral} structures on all objects:
$\mg$, strata defined in Sections~\ref{sect:invariants} and \ref{sect:para}, 
polarizing subalgebras etc. A precise choice of such structures is not essential 
for analytic considerations below; it suffices to 
observe that different choices of integral 
structures affect only finitely many places of $F$. In particular, they 
do not affect analytic properties of Euler products and height zeta functions.

\begin{nota}
\label{nota:gener}
Let $F$ be a finite extension of $\Q$. Denote by $F_v$
the completion of $F$ with respect to a valuations $v$; 
for $v$ nonarchimedean, denote by ${\mo}_v$ the ring of $v$-adic integers.
Denote by $\A:=\prod_v' F_v$ the ring of adeles of $F$. 
\end{nota}

\subsection{On $\mg$}
\label{sect:smg}

Let $\mg=\langle X_1,..., X_n\rangle $ 
be an $n$-dimensional Lie algebra 
over $F$, with a fixed  basis ${\bf X}$. 
Let $\mg_{\mo}'$ be the $\mo_F$-module 
$$
\mo_F X_1 +\dots + \mo_FX_n.
$$
There is an integer $a\in \Z$ such that 
$\mg_{\mo}=a\cdot \mg_{\mo}'$ is a  {\em Lie order}, i.e., 
a subring of $\mg$ all of whose coefficients lie in $\mo_F$.  
Indeed, write
$$
[X_i,X_j]=\sum c_{ij}^k X_k
$$
with $c_{ij}^k\in F$. Then 
$$
[aX_i,aX_j]=\sum (ac_{ij}^k)aX_k
$$
and we can choose $a\in \Z$ such that all $ac_{ij}^k\in \mo_F$.

\begin{defn}
A Lie order $\mg_{\mo}$ is 
called {\em admissible} if $\exp(\mg_{\mo})$ is a 
subgroup of $\bG(F)$. 
\end{defn}

Assume that $\mg$ is nilpotent. By 
the Baker-Campbell-Hausdorff formula
$$
X*Y:=\log(\exp(X)\cdot \exp(Y))= X+Y+\sum_{j=2}^{k}\rb_j(X,Y)
$$
where $X,Y\in \mg(F)$ and 
$\rb_j$ is a sum of 
$j$-fold brackets with coefficients in 
$F$. 

\begin{defn}
A Lie order $\mg_{\mo}$ is called {\em universal} if
$X,Y\in \mg_{\mo}$ implies that 
$$
\rb_{j}(X,Y)\in \mg_{\mo}
$$
for all $j\ge 2$. 
\end{defn}

Clearly, a universal Lie order is admissible.

\begin{lemm}
There exists an $a\in \Z$ such that 
$\mg_{\mo}=a\cdot\mg_{\mo}'$ is a universal order. 
\end{lemm}

\begin{proof}
For $X,Y\in \mg_{\mo}'$ we have
$$
(aX)*(aY)=\sum \rb_j(aX,aY)= aX +aY +\sum_{j\ge 2} \rb_j(X,Y)a^j.
$$
Now choose $a\in \Z$ such that for all $j\ge 2 $, 
$a^{j-1}$ times every coefficient of $\rb_j$ is in $\mo_F$. 
Then $a^{j-1}\rb_j(X,Y)\in \mg_{\mo}'$ and 
$$
\rb_j(aX,aY)\in a\mg_{\mo}'.
$$
\end{proof}

\begin{exam}
For $\mg=\mn_4$ (from Example~\ref{exam:exa}) we have
$$
X*Y=X+Y+\frac{1}{2}[X,Y]+
\frac{1}{12}[X,[X,Y]] -\frac{1}{12}[Y,[X,Y]].
$$
Then $6\cdot \mg_{\Z}$ is universal.
\end{exam}

Let $v$ be a nonarchimedean valuation. Write
$$
\mg(F_v):=\mg\otimes_F F_v,\,\,\, \mg(\mo_v):=\mg_{\mo}
\otimes_{\Z}\mo_v.
$$
If $\mg_{\mo}$ is universal then $\mg({\mo_v})$ is 
an admissible lattice in $\mg(F_v)$. In fact, if $o\in \mo_v$
then 
$$
\rb_j(X\times o,Y\otimes o)=\rb_j(X,Y)\otimes o^j\in \mg({\mo_v}).
$$
In particular, $\exp(\mg({\mo_v}))$ is a subgroup of $\bG(F)$.

\

We apply the preceding discussion as follows:
Let $\mg=\mg_F$ be a nilpotent Lie algebra over a number field
$F$. Fix a strong Malcev basis $(X_1,...,X_n)$ of $\mg$. 
Choose an $a\in \Z$
such that $\mg_{\mo}=a\cdot \mg_{\mo}'$ is a universal order in $\mg_F$. 
Then $a\mg({\mo_v}):=a\mg_{\mo}'\otimes_{\Z}\mo_v$ 
is an admissible lattice in $\mg_F$. 
Set $\bfK_v:=\exp(a\mg({\mo_v}))\subset \bG(F_v)$. 
This is a compact subgroup.

\begin{rem}
Once a representation 
$\rho_F\,:\, \bG\ra \bN$ as in Section~\ref{sect:uF} 
is fixed, we have 
$$
\bfK_v=\bG(F_v)\cap \bN(\mo_v)
$$ 
for almost all $v$
($\bN$ is defined over $\Z$).  
\end{rem}


\subsection{Measures}
\label{sect:measures}

We fix Haar measures  $dx_v$ on 
$\ga(F_v)$ for all $v$ (normalized as in \cite{tate67b}). 
For all but finitely many $v$ the volume of
$\ga(\mo_v)$ with respect to $dx_v$ is equal to 1.
Thus we have an induced Haar measure $dx$ on $\ga(\A)=\A$
and on $\ga^n(\A)$ for all $n$. 
Using the homeomorphism between  $\bG(\A)$ and 
$\ga^n(\A)$ 
we may realize the Haar measure $dg=\prod_v dg_v$ on $\bG(\A)$ as the
product measure 
$$
dg= dx_1\dots dx_{n}. 
$$
Similarly, we have $\vol(\bG(\mathfrak o_v)=1$, for almost all $v$. 
Let $\mm\subset \mg$ be a subalgebra and $\bM\subset \bG$ the corresponding 
subgroup. The induced integral structure on $\mm$ allows us to obtain a 
normalized Haar measure on $\bM(\A)$, again we have $\vol(\bM(\mo_v))=1$, 
for almost all $v$.

\begin{nota}
\label{nota:norm}
For each $\Sigma_{\bf d}\in \Sigma$ we fix a finite set
of $\Ad^*(G)$-invariant polynomials $P_{{\bf d},j}\in F[\mg^*]$ separating the orbits, 
as in Proposition~\ref{prop:strata}. 
Let $v\in S_{\infty}$ (the archimedean valuations of $F$),  
$\ell\in Z_{\bf d}(F)$ 
and $\cO_{\ell}$
be the corresponding $\Ad^*(G)$-orbit in $\mg^*(F_v)$. 
Define the norm of the orbit 
\begin{equation}
\label{eqn:norm}
\|\cO_{\ell}\|_{\infty}:=
\max_{v\in S_{\infty}}\max_{j\in \sJ_{\bf d}}
|P_{{\bf d},j}(\ell)|_v.
\end{equation}
\end{nota}

A priori, the definition of this norm depends on the choice of 
$\Ad^*(G)$-invariant polynomials separating the orbits in $Z_{\bf d}$. 
However, for any $\Ad^*(G)$-invariant polynomial $P\in F[\mg^*]$ 
there exists an $N=N(P)$ such that 
$$
\max_{v\in S_\infty} |P(\ell')|_{v} \le \|\cO_{\ell}\|_{\infty}^N,
$$
for all $\ell'\in \cO_{\ell}$. In particular, a norm as in \eqref{eqn:norm},
defined via a different choice of polynomials separating the orbits in $Z_{\bf d}$, 
will be comparable, up to powers. We have a fundamental finiteness result: 
let $\mathfrak l\subset \mathfrak o_F$ be any lattice. Then there exists an $n_0\in \N$
such that  
$$
\sum_{\ell \in \mathfrak g^*(\mathfrak l)/\Ad^*(G) } \|\cO_{\ell}\|_{\infty}^{-n}, 
$$
is convergent for all $n\ge n_0$.

\section{Representations: basics}
\label{sect:reprs}

In this section we describe Kirillov's orbit method
in the theory of unitary representations of nilpotent groups 
over local fields and its generalization to adeles by
Moore (see \cite{kirillov0} and \cite{moore}).

\

\subsection{The orbit method}
\label{sect:orb-method}

\begin{nota} 
\label{nota:general}
Let $F$ be a number field, $v$ a valuation and $F_v$ the
$v$-adic completion of $F$. Denote by 
$\mm_v\subset {\mathfrak o}_v$ the maximal ideal
in the ring of $\mo_v$ of $v$-adic integers
(for nonarchimedean $v$). 
We write $\k=\k_v$ for the residue field of ${\mo}_v$ and 
$q=q_v$ for the cardinality of $\k_v$.
We denote by $\Val(F)=\{|\cdot|_v\}=
S_{\rm fin}\cup S_{\infty}$ the set of 
all valuations of $F$, here $S_{\infty}$ is
the set of archimedean and $S_{\rm fin}$ the set of 
nonarchimedean valuations. 
We normalize the valuations 
in such a way that for any Haar measure $\mu_v$ on $F_v$ one
has $\mu_v(a{\cM})=|a|_v\mu_v({\cM})$ 
for all measurable subsets 
${\cM}\subset F_v$ and all $a\in F_v^*$.  
We continue to denote by $\A=\A_F$  the adele ring of $F$.
For any finite set $S\subset \Val(F)$ we put
$\A_S=\prod_{v\in S}F_v$, 
$\A^S=\prod_{v\notin S}' F_v$ (restricted product).
We abbreviate 
$\A_{\rm fin}=\A^{S_\infty}$ and $\A_{\infty}=\A_{S_{\infty}}$. 
\end{nota}

First we recall basic facts concerning harmonic
analysis on additive groups
(cf., for example, \cite{tate67b}).
For any prime number $p$, we have an embedding
$\Q_p/\Z_p\hookrightarrow \Q/\Z$. 
Using it we can define a (unitary) character $\psi_p$
of the additive group $\ga(\Q_p)$ by
\[
\psi_p \,: \, x_p \mapsto \exp(2 \pi i x_p).
\]
At the infinite place of $\Q$ we put
\[
\psi_{\infty}\colon  x_{\infty} \mapsto \exp(-2 \pi i x_{\infty}), 
\]
(here $x_{\infty}$ is viewed as an element in $\R/\Z$).
Taking the product we get a character
$\psi$ of $\ga(\A_{\Q})$ and, by composition with the trace,
a character $\psi=\psi_1=\prod_v\psi_v$ of $\ga(\A)$.
This defines a Pontryagin duality 
$$
\ga(\A)\ra\ga(\A)^*
$$
$$
(a_v)\mapsto ((x_v)\ra \prod_v \psi_v(a_vx_v)).
$$
The subgroup $\ga(F)\subset \ga(\A)$ is 
discrete, cocompact and selfdual under the above duality.

\

Denote by $\mg^*$ the dual to the Lie algebra
$\mg$ of $\bG$. It inherits the 
$F$-rational structure from $\mathfrak g$.  
For every $F$-rational linear form $\ell\in \mg^*(F)$
let ${\mm}_{\ell}$ be a polarizing to $\ell$ subalgebra
of $\mg$ (see Section~\ref{sect:nill}). 
Then $\ell$ defines a character on the adelic points 
$\bM_{\ell}(\A)$ of the subgroup 
$\bM_{\ell}=\exp({\mm}_{\ell})\subset \bG$ 
$$
\psi_{\ell} =\psi_1\circ \ell\circ \log\,:\, \bM_{\ell}(\A)\ra 
{\mathbb S}^1\subset \C^*.
$$
Let  
$$
\pi_{\ell}=\Ind_{\bM_\ell(\A)}^{\bG(\A)}(\psi_\ell)
$$
be the induced representation. 
Then 
\begin{itemize}
\item $\pi_\ell$ is irreducible;
\item $\pi_\ell$ does not depend on the choice of ${\mm}_\ell$
(up to isomorphy);
\item $\pi_\ell$ does not depend on the choice of $\ell$ in 
the $\Ad^*$-orbit $\cO_{\ell}$ (up to isomorphy).
\end{itemize}
This is the orbit picture proposed by Kirillov: 
irreducible unitary representations
of a unipotent group $\bG$ 
are parametrized by orbits of the coadjoint 
action of $\bG$ on $\mg^*(F)$.

\begin{nota}
\label{nota:ginf}
Let $\bG_{\infty}:=\prod_{v\in S_{\infty}}\bG(F_v)$
and $\Gamma$ be a discrete cocompact subgroup in $\bG_{\infty}$
(e.g., the image of $\bG(\mathfrak l)$, where $\mathfrak l\subseteq \mo_F$ is a
sublattice). 
\end{nota}

\begin{nota}
\label{nota:l2}
Denote by 
$$
\cH:=\sL^2(\bG(F)\backslash \bG(\A))
$$ 
the space
of (left) $\bG(F)$-invariant square-integrable 
(on the quotient) functions on $\bG(\A)$ (and similarly 
$\sL^2(\Gamma \backslash \bG_{\infty})$).
Denote by 
$$
\cH^{\bfK}=\sL^2(\bG(F)\backslash \bG(\A))^{\bfK}
$$ 
the subspace of (right) $\bfK$-fixed
vectors.
\end{nota}

\begin{prop}\cite{moore}
\label{prop:decomp}
For each $\ell\in \mg^*$ there is a unitary equivalence between $\pi_{\ell}$ and $\varrho_{\ell}$ implemented by 
an isometry
$$
\begin{array}{ccccl}
j_{\ell} & : & 
\Ind_{\bM_{\ell}(\A)}^{\bG(\A)}(\psi_{\ell}) & 
\rightarrow & \cH_{\ell}\subset \cH
\end{array}
$$
given by 
$$
\phi(x)   \mapsto \sum_{\gamma\in \bM_{\ell}(F)\ba \bG(F)} \phi(\gamma \cdot x), 
$$
where $(\varrho_{\ell},\cH_{\ell})$ is a unitary irreducible representation
occurring in the right regular representation of $\bG(\A)$ on $\cH$. 
Moreover, $j_{\ell}$ 
induces an isometry on the subspaces of  $\bfK$-fixed vectors
and 
$$
\cH^{\bfK}=\oplus_{\ell\in {\mg}^*(F)/{\Ad}^*} \cH_{\ell}^{\bfK},
$$
as a direct sum of irreducible unitary representations of $\bG(\A)$, 
each occurring with multiplicity one.  
\end{prop}

\begin{nota}
\label{nota:RK}
We denote by $\cR(\bfK)=\{ \varrho\}$ the set of 
irreducible unitary representation of $\bG(\A)$ occcuring in the right action of $\bG(\A)$ on $\cH^{\bfK}$.
\end{nota}

\subsection{Multiplicities}
\label{sect:mult}

Let $F$ be a number field and $S_G\subset S_{\rm fin}$ 
a finite set of nonarchimedean valuations 
of $F$ such that for all $v\notin S_G$ one has 
$$
\exp(\mg(\mo_v))=\bG(\mo_v)=\bfK_v.
$$
For $v\in S$ the compacts $\bfK_v$ are defined as
in  Section~\ref{sect:integral}.

\begin{nota}
Let $\pi_v$ be a unitary 
representation of $\bG(F_v)$. 
We denote by 
$$
m(\pi_v,\bfK_v,{\bf 1})
$$
the multiplicity of the trivial representation
${\bf 1}$ occurring in the restriction of $\pi_v$ to $\bfK_v$. 
\end{nota}

\begin{prop}
\label{prop:mult}
Let $\ell\in \mg^*(F_v)$ and $\pi_v$ be the unitary irreducible
representation of $\bG(F_v)$ 
corresponding to the orbit $\cO_{\ell}$. 
Then there exist constants 
$\sc_v$, with $\sc_v=1$ for all $v\notin S_G$ such that
$$
m(\pi_v,\bfK_v,{\bf 1})\le \sc_v|\Pf(\ell)|_v^{-1},
$$
where $\Pf(\ell)$ is the Pfaffian (defined in Section~\ref{sect:pfaf}), evaluated 
at $\cO_\ell$. 
\end{prop}

\begin{nota}
\label{nota:spi}
Let $S_{\pi}$ be the set of nonarchimedean $v$ such that either $v\in S_G$ or 
$m(\pi_v,\bfK_v,\mathbf 1)\neq 1$. 
\end{nota}

\begin{rem}
In \cite{howe-1}, \cite{richardson}, \cite{fox} one can find
bounds for $m(\pi_v,\bfK_v,{\bf 1})$ in terms of the 
the number of $\Ad^*(\bfK_v)$-orbits on $\cO_{\ell}(\bfK_v)$. 
An estimate in terms of Pfaffians has been derived
in \cite{cg-2}.
\end{rem}

\noindent
For completeness, we include a proof of this proposition. 

\

We choose the Haar measure $dg_v$ 
on $\bG(F_v)$ so that
$$
\int_{\bfK_v}dg_v=1.
$$
We normalize the measure $dX_v$ on $\mg(F_v)$ such that
$$
\int_{\bG(F_v)}\phi(g_v)dg_v=\int_{\mg(F_v)}\phi(\exp(X_v))dX_v
$$
for all smooth compactly supported functions
$\phi\in \sC^{\infty}_c(\bG(F_v))$.  

For $\ell=\ell_v\in \mg^*(F_v)$ 
denote by $d\mu(\ell)=d\mu_v(\ell)$ 
the (canonical) $v$-adic measure on the ($2k$-dimensional)
orbit $\cO_{\ell}(F_v)\subset \mg^*(F_v)$ 
associated with the top degree form 
$\mu(\ell):=\wedge^k\Omega_{\ell}$, 
where $\Omega_{\ell}$ is the {\em canonical}
invariant algebraic 2-form on $\cO_{\ell}$. 

For appropriate functions $\phi$ on $\bG(F_v)$ 
define the Fourier transform 
$$
\hat{\phi}(\ell):=
\int_{\mg(F_v)}\phi(\exp(X_v))\psi(\langle \ell,X_v\rangle)dX_v.
$$
Let $\pi_v$ 
be an irreducible unitary representation of $\bG(F_v)$. 
The function $\phi$ defines the operator
$$
\pi_v(\phi):=\int_{\bG(F_v)} \phi(g_v)\pi_v(g_v)dg_v.
$$
It is of trace class.

\begin{lemm}
\label{lemm:tr}
Let $v$ be a nonarchimedean valuation and 
$\pi_v$ an 
irreducible unitary representation of $\bG(F_v)$
corresponding to $\ell\in \mg^*(F_v)$. 
Let $\phi\in \sC_c^{\infty}(\bG(F_v))$. Then 
$$
\tr\, \pi_v(\phi)=\frac{1}{2^kk!}
\int_{\cO_{\ell}}\hat{\phi}(\ell)d\mu_v(\ell)
$$
\end{lemm}

\begin{proof}
See \cite{cg}, p. 145.
\end{proof}

Let $\chi=\chi_v$ be the characteristic function of $\bfK_v$. 
Then the convolution $\chi*\chi=\chi$. Thus, 
$\pi_v(\chi)$ is a self-adjoint projection on $\cH_{\pi_v}$.
Moreover, 
$$
\tr(\pi_v(\chi))=m(\pi_v,\bfK_v,{\bf 1})
$$
and 
$$
\hat{\chi}(\ell_v)=\int\chi_0(X_v)
\psi(\langle \ell_v,X_v \rangle)dX_v
$$
where $\chi_0$ is the characteristic function of $\mg(\mo_v)$. 
It follows that $\hat{\chi}$ is the characteristic function of
the dual lattice $\mg^*(\mo_v)\subset \mg^*(F_v)$. 
Therefore, 
\begin{equation}
\label{eqn:mmm}
m(\pi_v,\bfK_v,{\bf 1})=\frac{1}{2^kk!}\int_{\cO_{\ell}}
\hat{\chi}(\ell)d\mu_v(\ell),
\end{equation}
where $d\mu_v(\ell)$ is the canonical measure on the orbit
$\cO(\ell)$. 

The Lie algebra $\mg$ is equipped with a fixed 
strong Malcev basis $\langle X_1, ..., X_n\rangle$. 
The dual basis $\langle \ell_1,...,\ell_n\rangle$ 
in $V:=\mg^*$ is a Jordan-H\"older basis. Recall the 
stratification of representations explained in 
Section~\ref{sect:para}.
Assume that the representation 
$\pi_v$ belongs to the stratum 
$\Sigma_{\bf d}$ (as in Section~\ref{sect:para}). 
Denote by 
$$
V^{\bf d}_{\sI}:=\langle \ell_i; i\in \sI_{\bf d}\rangle_{F_v},\,\,\,
V^{\bf d}_{\sJ}:=\langle \ell_j; j\in \sJ_{\bf d}\rangle_{F_v}
$$ 
the affine subspaces defined in \ref{prop:strata}.
Recall  that 
$V^{\bf d}_{\sI}=\langle \ell_{i_1},...,\ell_{i_{2k}}\rangle$ 
is even dimensional.    
Regarding $X_{i_1},...,X_{i_{2k}}$ as
(independent) linear forms on $V^{\bf d}_{\sI}$
define a Haar measure on $V^{\bf d}_{\sI}$ by
$$
|d\mu|:=|X_{i_1}\wedge \cdots \wedge X_{i_{2k}}|.
$$ 
Let 
$$
\begin{array}{ccccc}
f_{\ell} &: &  V^{\bf d}_{\sI}& \ra &  \cO_{\ell}\\
         &  &    u    & \mapsto & (u,P_{\ell}(u))
\end{array} 
$$ 
be the map parametrizing the orbit $\cO_{\ell}$
(here $P_{\ell}$ is a polynomial on $V^{\bf d}_{\sI}$).
Then 
$$
\tilde{\mu}(\ell)
:=(df_{\ell}^*)^{-1}(X_{i_1}\wedge \cdots \wedge X_{i_{2k}})
$$
is a $\bG$-invariant volume form on the orbit $\cO_{\ell}$. 
Therefore, 
$$
\mu(\ell)=\sc(\ell)\tilde{\mu}(\ell)
$$
(and $\sc(\ell)$ depends only on the orbit $\cO_{\ell}$). 
Denote by $\mo_{v,\sI}^*$ the image of the projection 
$$
\pr_{v,\sI}\,:\, \mg^*(\mo_v)\ra V_{v,\sI}
$$
and by $\chi_{v,\sI}^*$ the characteristic function of 
this set. Continuing from (\ref{eqn:mmm}), we obtain  
\begin{align}
m(\pi_v,\bfK_v,{\bf 1}) & = \frac{|\sc(\ell)|}{2^kk!}
\int_{\cO_{\ell}} \hat{\phi}(\ell)|d\tilde{\mu}_v(\ell)|\\
                        & = \frac{|\sc(\ell)|}{2^kk!}
\int_{V_{v,\sI}}\hat{\phi}(u,P_{\ell}(u))|d\mu_v(\ell)|\\
                        & \le \frac{|\sc(\ell)|}{2^kk!}
\int_{V_{v,\sI}}\chi_{v,\sI}^*(u)|d\mu_v(\ell)|
\end{align}

We fix an integer $a\in \Z$ so that $\mg_\mo$ is
equal to the $\mo_F$-span of $\{aX_1,...,aX_n\}$
(replacing $X_j$ by $a^{-1}X_j$ we may assume from 
the beginning that 
$$
\mg_\mo= \mo_FX_1\oplus ...\oplus \mo_FX_n.
$$ 
Then
$$
\begin{array}{ccc}
\mg^*(\mo_v) & = & \mo_v\ell_1\oplus ...\oplus\mo_v\ell_n \\
\mo_{v,\sI}^* & = & \mo_v\ell_{i_1}\oplus ...\oplus \ell_{i_{2k}}. 
\end{array}
$$
The Haar measure is normalized such that $\vol(\mo_v)=1$. 
Then 
$$
\int_{V_{v,\sI}}\chi_{v,\sI}^*(u)|d\mu_v(\ell)|=1.
$$
It remains to observe that 
$$
\sc(\ell)= \Pf(\ell)^{-1}.
$$

\subsection{Spherical functions}
\label{sect:spherical}

\begin{prop}
\label{prop:spherical}
For $\ell\in \mg^*(F)$ let $\mm_{\ell}\subset \mg$ be a
polarizing subalgebra, $\bM=\bM_{\ell}=\exp(\mm_{\ell})$, 
and $\psi=\psi_\ell=\psi_1\circ \ell$ the corresponding 
adelic character of $\bM(\A)$. 
Let $\cH_{\ell}$  be the associated
irreducible unitary representation of $\bG(\A)$. 
Let $\bfK=\prod_{v\notin S_{\infty}} \bfK_v$ 
be as in Proposition~\ref{prop:height} and 
assume that $\cH_{\ell}^{\bfK}\neq 0$. 
Then for all $\omega \in \cH_{\ell}^{\bfK}$, 
with $\|\omega\|_{\sL^2}=1$, all 
$v\notin S_{\varrho_{\ell}}$ and all (integrable)
functions $H_v$ on $\bG(F_v)$ 
such that 
$$
H_v(k_vg_v)=H_v(g_v),
$$ 
for all $k_v\in \bfK_v, g_v\in \bG(F_v)$, one has 
$$
\int_{\bG(F_v)}H_v(g_v)\omega_v(g_v)dg_v=
\int_{\bM(F_v)}H_v(h_v)\psi_{v}(h_v)dh_v,
$$
(where $dh_v$ is normalized as in Section~\ref{sect:measures}).
\end{prop}

\begin{proof}
Define the function 
$$
\tilde{\psi}=\prod_v\tilde{\psi}_v\in 
\Ind_{\bM(\A)}^{\bG(\A)}(\psi)=:\pi
$$ 
as follows:
$$
\begin{array}{cccl}
\tilde{\psi}_v(g_v)   &  = & 0 & 
{\rm if}\,\,  g_v\notin \bM(F_v)\bfK_v\\
\tilde{\psi}_v(h_vk_v)&  = & \psi_v(h_v) & {\rm otherwise}
\end{array}  
$$
For all $v\notin S_{\pi}$ we have
$$
\psi_v|_{\bM(F_v)\cap \bfK_v}=1.
$$
By definition, for $v\notin S_{\pi}$ the representation
$\pi_v$ has a unique $\bfK_v$-fixed vector (of norm 1). 
A direct computation shows that 
$$
\|\tilde{\psi}\|_{\sL^2(\bM(\A)\ba \bG(\A))}=1.
$$
Therefore, the (local) spherical function  
$\varphi_v$ (normalized by $\varphi_v(e)=1$) is given by
$$
\varphi_v(g_v)=
\langle \pi_v(g_v)\tilde{\psi}_v, \tilde{\psi}_v\rangle.
$$ 
Now we compute:
\begin{align*}
\int_{\bG(F_v)}H_v(g_v)\varphi_v(g_v)dg_v & = 
\int_{\bG(F_v)}\int_{\bfK_v\cap \bM(F_v)\ba \bfK_v}H_v(k_vg_v)\tilde{\psi}_v(k_vg_v)dk_v'dg_v \\
 & = 
\int_{\bfK_v\cap \bM(F_v)\ba \bfK_v}dk_v'
\int_{\bG(F_v)} H_v(k_vg_v)\tilde{\psi}_v(k_vg_v)dg_v \\
& = \vol \cdot 
\int_{\bfK_v\ba \bG(F_v)}H_v(g_v)\tilde{\psi}_v(g_v)dg_v\\
&  = \vol \cdot \int_{\bM(F_v)\cap \bfK_v\ba \bH(F_v)}
H_v(h_v)\tilde{\psi}_v(h_v)dh_v\\
& = \vol \cdot \int_{\bM(F_v)}
H_v(h_v)\psi_v(h_v)dh_v.
\end{align*}
Here $dk_v'$ is the induced 
measure and $\vol=\vol(\bfK_v\cap \bM(F_v)\ba\bfK_v)$.
\end{proof}

\section{Universal enveloping algebra}
\label{sect:ug}

Now we turn to archimedean places. 

\subsection{Basics}
\label{sect:basu}

Let 
$$
\mT(\mg):=\oplus_{j\ge 0} \mg^{\otimes j}
$$ 
be the tensor algebra, 
$$
\mS(\mg):=\mT(\mg)/\langle X\otimes Y-Y\otimes X\rangle
$$ 
the symmetric algebra and 
$$
\mU(\mg):=\mT(\mg)/\langle X\otimes Y-Y\otimes X -[X,Y]\rangle
$$ 
the universal enveloping algebra 
of $\mg$.
There is an injective map
$$
\mg\ra \mT(\mg)\ra \mU(\mg)
$$
and a $\mg$-module isomorphism (symmetrization)
$$
\sym\,:\, \mU(\mg)\ra \mS(\mg),
$$
which is defined on monomials by
$$
Y_1\cdots Y_r\mapsto\frac{1}{r!} 
\sum_{\sigma\in {\mathbb S}_r}Y_{\sigma(1)}\cdots Y_{\sigma(r)}
$$ 
(where ${\mathbb S}_n$ is the symmetric group).
Each $Y\in \mg$ defines a differential operator
$$
\partial_Y\,:\, f(g)\mapsto \frac{{\rm d}}{{\rm d}t}
f(g\cdot \exp(tY))|_{t=0}
$$
on smooth functions on $\bG(F_v)$,
for any  archimedean $v$. This gives
a surjective algebra homomorphism from $\mU(\mg)$
onto the algebra of left-invariant 
differential operators on $\sC^{\infty}(\bG(F_v))$. 
In particular, $\mU(\mg)$ acts in the space of
smooth vectors of every irreducible unitary representation
$(\varrho,\cH)$ of $\bG(F_v)$. 
For $\partial \in \mU(\mg)$ we will denote by 
$\varrho(\partial)$ the corresponding operator.

\

We will use the canonical identification  
$\mS(\mg)=F[\mg^*]$ (by duality):
$$
Y\ra f_{Y}\in F[\mg^*],\,\,\, 
f_Y(\ell)=\ell(Y);\,\,\, Y\in \mg, \ell\in \mg^*.
$$
The adjoint action of $\bG$ on $\mg$ extends to actions of $\bG$ on 
to $\Aut_F(\mU(\mg))$ and $\Aut_F(\mS(\mg))$.

\begin{lemm}
\label{lemm:symm}
The symmetrization 
$\sym$ is equivariant with respect to the adjoint 
action of $\bG$ and maps the space of  
$\Ad^*$-invariant polynomials on $\mg^*$ 
bijectively onto the center $\mZ\mU(\mg)$ of
$\mU(\mg)$. 
\end{lemm}

\

\subsection{Scalar operators}
\label{sect:scalars}

\begin{prop}
\label{prop:scalar}
Let $v$ be an archimedean valuation, $\cO\subset 
\mg^*$ a coadjoint orbit and
$$
(\varrho_{\ell},\cH_{\ell})
\sim\Ind_{\bM_{\ell}(F_v)}^{\bG(F_v)}(\psi_{\ell})
$$
for some $\ell\in \cO $ and some polarizing $\bM_{\ell}$. 
For $P\in F[\mg^*]$ let $\partial_P\in \mU(\mg)$ be the 
corresponding differential operator. 
Assume that the restriction of $P$ to 
$\cO$ is identically constant
$$
P(2\pi i \ell)=P(2 \pi i\ell')
$$  
for all $\ell,\ell'\in \cO$.
Then
$$
\varrho_{\ell}(\partial_{P})f= P(2\pi i\ell) \cdot f
$$
for all smooth vectors $f\in \cH_{\ell}$ and all $\ell\in \cO$.

In particular, let $\partial_z\in \mZ\mU(\mg)$ and 
$P_z$ be the corresponding $\Ad^*$-invariant
polynomial (see Lemma~\ref{lemm:symm}). 
Then, for all orbits $\cO$ and all $\ell\in \cO$, 
the operator $\varrho_{\ell}(\partial_z)$ acts in 
$\cH_{\ell}$ by multiplication by 
$$
P_z(2\pi i\ell).
$$
\end{prop}

\begin{proof}
We follow closely 
the exposition in \cite{cg}, p. 186.
The proof proceeds by induction on the dimension of $\mg$. 
We explain the case when $v$ is real, the complex places
being similar. 

\

Assume that there is a nontrivial ideal $\mz_0\subset \mz_{\mg}$
such that $\varrho_{\ell}$ 
restricted to $\exp(\mz_0)$ is trivial.
Consider the projections
$$
\begin{array}{cccl}
\pr: &  \mg & \ra &  \mg_0:= \mg/\mz_0,\\
\pr: & \bG  & \ra & \bG_0.
\end{array} 
$$
The induced injection
$$
{\it in}\,:\,\mg_0^*\hookrightarrow \mz_0^{\perp}\subset \mg^*
$$
maps $\cO$ isomorphically onto $\cO_0$. 
The maps 
are equivariant with respect to the (co)adjoint actions of
$\bG$ and $\bG_0$. They extend naturally 
to symmetric algebras, universal enveloping 
algebras and polynomial functions. In particular, 
$$
\pr \,:\, F[\mg^*]\ra F[\mg_0^*]
$$
is simply the restriction of the polynomial $P\in F[\mg^*]$
to $\mz_0^{\perp}=\mg_0^*$. 
The representations
$\varrho_{\ell}$ and 
$\varrho_{\ell_0}=\varrho_{\ell} \circ \pr$
correspond to the same orbit 
$$
\cO=\cO_0\subset \mg_0^*\subset \mg^*.
$$  
We have an equivariant commutative diagram
$$
\begin{array}{ccccc}
F[\mg^*] & \simeq & \mS(\mg) & \stackrel{\sym}\longleftarrow & \mU(\mg)\\
\downarrow &      &  \downarrow &                      & \downarrow\\
F[\mg_0^*] & \simeq &\mS(\mg_0)&
\stackrel{\sym}\longleftarrow & \mU(\mg_0).
\end{array}
$$
Let $\ell\in \mg^*$, let $\bM_{\ell}$ 
be a polarizing subgroup for $\ell$, and $\varrho_{\ell}$ the corresponding
irreductible unitary representation of 
$\bG(F_v)$ as in Section~\ref{sect:reprs}. 
We have 
$$
\varrho_{\ell}(\partial_{P})=\varrho_{\ell_0}(\partial{\pr(P)})=
\pr(P)(2\pi i \ell_0)\cdot \Id_0 = P(2\pi i \ell) \cdot \Id,
$$
as claimed.

\

Now assume that $\dim \mz_{\mg} =1$ and that
$\varrho_{\ell}$ is nontrivial on $\mg$ ($\ell(Z)\neq 0$). 
Choose a reducing quadruple as in Kirillov's lemma: 
$$
\mg=\mg_0\oplus FX
$$ 
(\ref{nota:redq}). 
This time we have an injection 
$$
{\it in}\,:\, \mg_0\hookrightarrow \mg
$$
and an induced projection 
$$
\pr\,:\, \mg^*\ra \mg_0^*.
$$
We have
$$
\ell_0:=\pr(\ell)=\ell|_{\mg_0}.
$$
By Lemma~\ref{lemm:restr}, we have 
$$
\pr(\cO_{\ell})=\sqcup_{t\in \R} \cO_{\ell_t},
$$
where 
$$
\cO_{\ell_t}:=\{ \bG_0\cdot \ell_{t}\}\Ad^*(\exp(tX))(\ell_0).
$$
The $\Ad^*$-invariance
of $P$ at $\ell\in \cO$ implies that
the restriction of $P$ to $\cO$ does not depend on $X$. 
In particular, the restriction of $P$ to each $\cO_{\ell_t}$
is invariant under the adjoint action of $\bG_0$. 

We have a direct integral
decomposition 
$$
\varrho=\int_{\R}^{\oplus}\varrho_{t}dt,
$$
where $\varrho_t$ is the unitary irreducible representation
of $\bG_0(F_v)$ associated to the orbit $\cO_{\ell_t}$. 
This decomposition 
passes to smooth vectors (see \cite{cg}, p. 188).
By the induction hypothesis, 
$$
\begin{array}{ccl}
\varrho_{t}(\pr(P))f & = & \pr(P)(2\pi i \ell_t)f \\
                     & = & \pr(P)(2\pi i (\Ad^*(\exp(tX))\ell_0))f\\
                     & = & P(2\pi i (\Ad^*(\exp(t X))\ell))f\\
                     & = & P(2\pi i \ell)f
\end{array}
$$ 
for all smooth vectors $f$ in the representation
space of $\varrho_t$, since the 
projection $\pr \,:\, \mg^*\ra\mg_0^*$ is equivariant for 
the coadjoing action of $\bG$, resp. $\bG_0$, and $P$ is 
invariant. Further, since $\varrho_0(\pr(P))=\varrho(P)$, 
as elements in $\mU(\mg)$, $\varrho(P)$ is 
determined by the restriction of $\varrho$ to $\bG_0$
and 
$$
\varrho(P)=\int_{\R}^{\oplus}
\varrho_{t}(\pr(P))dt=P(2\pi i \ell)\Id.
$$
\end{proof}

\section{Geometry}
\label{sect:geom}

Here we work over an arbitrary field $F$ of characteristic zero.

\begin{nota}
For a smooth projective algebraic variety $X$ over $F$ we
denote by $\Pic(X)$ its Picard group and by 
$\La_{\rm eff}(X)\subset \Pic(X)_{\R}$
the (closed) cone of pseudo-effective divisors on $X$.
We will often identify line bundles, the corresponding divisors 
and their classes in $\Pic(X)$. 
We write $\cL=(L,\|\cdot\|)$ when we want to 
emphasize that the line bundle $L$ is adelically metrized. 
If $X$ has an action by
a group $\bG$ we write $\Pic^{\bG}(X)$ for the group of
isomorphism classes of $\bG$-linearized line bundles on $X$.  
\end{nota}

\subsection{Main invariants}
\label{sect:invariants}

Let $X$ be a smooth projective variety with $-K_X$ 
contained in the interior of the effective cone. Then $\Lambda_{\rm eff}(X)$ is a rational finitely 
generated cone, by \cite{BCHM}. 
In this case, given a line bundle $L$ on $X$, we let
$$
a(L):=\inf\{ a\,|\, a[L]+[K_X]\in \La_{\rm eff}(X)
$$
and $b(L)$ be the codimension of the face of $\La_{\rm eff}(X)$
containing $a(L)[L]+[K_X]$. 
When $X$ is {\em singular}
and $\rho\,:\, \tilde{X}\ra X$ a desingularization
satisfying the conditions above, we can define
$$
a(L):=a(\rho^*(L)),\,\,\, b(L):=b(\rho^*(L)).
$$ 
These invariants are well-defined (see \cite{HTT} for more details).

\begin{prop} 
\label{prop:geometry}
Let $X$ be a smooth 
projective equivariant compactification of a 
unipotent algebraic group $\bG$. 
Let $D:=X\setminus \bG$ be the boundary and  
$(D_{\al})_{\al\in \cA}$ the set of 
its irreducible components. Then 
\begin{itemize}
\item  $\Pic^{\bG}(X)=\Pic(X)$;
\item  $\Pic(X)$ is freely generated by the classes $D_{\al}$;
\item  $\La_{\rm eff}(X)=\oplus_{\al} \R_{\ge 0} D_{\al}$;
\item  $-K_X=\sum_{\al}\kappa_{\al}D_{\al}$
with $\kappa_{\al}\ge 2 $ for all $\al\in \cA$. 
\end{itemize}
\end{prop}

\begin{proof}
Analogous to the proofs in Section 2 of \cite{HT}.
In particular, it suffices to assume that $X$ carries only 
a one-sided action of $\bG$. Notice that every line bundle
admits a unique $\bG$-linearization.    
\end{proof}

\begin{coro}
\label{coro:e}
The divisor of every  
irreducible polynomial 
$$
f\in F[\bG]=F[x_1,...,x_n]
$$ 
can be written as 
$$
\dv(f)=E(f)-\sum_{\al} d_{\al}(f)D_{\al}
$$
where $E(f)$ is the unique irreducible component of 
$\{f=0\}$ in $\bG$ and $d_{\al}(f)\ge 0$ for all $\al$. 
\end{coro}

\begin{prop}
\label{prop:m}
Let $X$ be a smooth equivariant compactification of a 
unipotent group $\bG$ (with Lie algebra $\mg$).
Let $\mm\subset \mg$ be a subalgebra and $Y:=Y_{\mm}\subset X$
the compactification of $\bM=\exp(\mm)\subset \bG$. 
Then 
$$
(a(-K_X|_Y),  b(-K_X|_Y)) < (a(-K_X), b(-K_X)),
$$ 
in the lexicographic ordering. In particular, the set of pairs
$$
(a(-K_X|_{Y}), (b(-K_X|_{Y}))
$$
is finite, as $\mm$ ranges over the set of all subalgebras.
\end{prop}

\begin{proof}
In the additive case when $\bG=\mathbb G_a^n$, this is the content of Lemma 7.3 in \cite{CLT}.
The case of general linear groups is covered in \cite{HTT}. 
\end{proof}

\subsection{Uniformity}
\label{sect:uni}

Let $X$ be a smooth projective equivariant compactification of 
$\bG$ and $\mm\subset \mg$ a subalgebra. Denote by $Y=Y_{\mm}$
the Zariski closure of $\exp(\mm)$ in $X$. It is a 
compactification of $\bM=\exp(\mm)$, not necessarily smooth. 
Denote by $D=D_{\mm}$ the boundary $Y\setminus \bM$.  

\begin{prop}
\label{prop:uni-deg}
There exist constants
$\sd,\sd',\sn>0$ such that 
for every subalgebra $\mm\subset \mg$ there exists am equivariant
blow-up $\tilde{Y}_{\mm}$ with support in 
the boundary $Y_{\mm}\setminus \exp(\mm)$ such that 
\begin{itemize}
\item $\tilde{Y}_{\mm}$ is smooth and projective;
\item the boundary of $\tilde{Y}_{\mm}$
is a strict normal crossings divisor; 
\item the number of boundary components 
is bounded by $\sn$;
\item the degree of every boundary 
component of $\tilde{Y}_{\mm}$ 
is bounded by $\sd$;
\item for every linear subspace $\me\subset \mm$ 
the degree of the intersection of the Zariski closure
in $\tilde{Y}_{\mm}$ of $\exp(\me)$  
with every boundary component is bounded by $\sd'$.   
\end{itemize}
\end{prop}

\begin{proof}
Noetherian induction. For each $k\le \dim \mg$ 
consider the Grassmannian $\GR$ of $k$-planes in $\mg$. 
This induces an algebraic family $\cY\ra \GR$ 
of subvarieties $Y_{\mm}\subset X$ and a family $\cD\subset \cY$
of boundary divisors. Taking a (finite) flattening stratification
of the base we reduce to the case when 
$\cY\ra B,\cD\ra B$ are flat over $B$. Now we use
embedded resolution of singularities over 
the function field of $B$.  Next we complete (projectively) to a family  
$(\tilde{\cY},\tilde{\cD})$ over $B$  and restrict to a Zariski open subset $B^0$ where 
both $\tilde{\cY}^0,\tilde{\cD}^0$ are flat over $B^0$, $\tilde{\cY}^0$
is smooth and $\cD^0$ is strict normal crossings.   
We repeat the process for each 
irreducible component of the complement to $B^0$. 

Each $\mm\in \mg$ will belong to one of the 
finitely many families constructed above.
For each family we can find uniform bounds as claimed. 
\end{proof}




\section{Height zeta function}
\label{sect:metrics}

\subsection{Heights}

Let $X$ be a smooth projective algebraic variety over
a number field $F$. 
A smooth adelic metrization of a line bundle $L$ on $X$
is a family of $v$-adic norms $\|\cdot \|_v$ on $L\otimes_{F_v}F$
for all $v\in \Val(F)$ such that
\begin{itemize}
\item  for $v\in S_{\infty}$ the norm $\|\cdot \|_v$ is  
$\sC^{\infty}$;
\item  for $v\in S_{\rm fin}$ the norm 
of every local section of $L$ is locally constant in 
the $v$-adic topology;
\item  there exist a finite set $S\subset \Val(F)$,
a flat projective scheme (an integral model) 
$\mathcal X$ over $\Spec(\mathfrak o_S)$
with generic fiber $X$ together with a line bundle $\cL$ on 
${\mathcal X}$ such that for all $v\notin S$ the 
$v$-adic metric is given by the integral model.
\end{itemize}

If ${\mathcal X}$ carries an action of an integral model 
${\mathcal G}$ of $\bG$ over an open dense subset
of $\Spec(\mathfrak o_F)$
extending the action of $\bG$ on $X$ and the line 
bundle $\cL$ has a ${\mathcal G}$-linearization
extending the $\bG$-linearization of $L$ then 
we call the smooth adelic metrization 
equivariant.    

\begin{prop}
\label{prop:height}
Let $\bG$ be a unipotent
algebraic group defined over a number field $F$ 
and $X$ a smooth projective  
bi-equivariant compactification of 
$\bG$.  
Then there exist a compact open subgroup 
$$
\bfK=\prod_v \bfK_v\subset \bG(\A_{\rm fin})
$$ 
and a height pairing
$$
H=\prod_{v\in \Val(F)} H_v\, :\, \Pic(X)_{\C}\times \bG(\A)\ra \C
$$
such that
\begin{itemize}
\item 
for all $L\in \Pic(X)$ the restriction of 
$H$ to $L\times \bG(F)$ 
is a height corresponding to some smooth adelic
metrization of $L$;   
\item 
the pairing is exponential in the $\Pic(X)$ component:
$$
H_v({\bf s}+{\bf s}';g)=H_v({\bf s};g)H_v({\bf s}';g)
$$
for all ${\bf s},{\bf s}'\in \Pic(X)_{\C}$, 
all $g\in \bG(\A)$ and all $v\in \Val(F)$; 
\item  for all $v\in S_{\rm fin}$  one 
has $H_v({\bf s};kgk')=H_v({\bf s} ;g)$ 
for all ${\bf s}\in \Pic(S)_{\C}$ and $k,k'\in \bfK_v$.  
\end{itemize}
\end{prop}

\begin{proof}
We follow closely the proof of Lemma 3.2 in \cite{CLT}.
First we observe that in our situation 
(with an action of $\bG\times \bG$)
$$
\Pic^{\bG\times \bG}(X)=\Pic(X).
$$ 
Let $\cL=(L,\|\cdot\|_v)$ be a very ampe line bundle
equipped with a locally constant $v$-adic norm (where
$v\notin S_{\infty}$). 
The space of $F$-rational global sections
$H^0(X,L)$ contains a unique (upto multiplication by 
$F^*$) $\bG\times \bG$-invariant 
section $f$. 
Consider the morphism of (left) multiplication
$$
m\,:\, \bG\times X\ra X
$$
and the projection
$$
pr\,:\, X\ra X.
$$
The trivial line bundle $m^*(\cL)\otimes pr^*(L)$
carries the tensor product metric.
Restricting to $\bG(F_v)\times \bG(F_v)$, we find
that the norm of the canonical section 1 is given by
$$
(g,x)\mapsto \|f(gx)\|_v\|f(x)\|^{-1}_v.
$$
It extends to a locally constant function on 
$\bG(F_v)\times X(F_v)$. 
Since it is locally constant and equal to 1 
on $\{1\}\times X(F_v)$ there exists a
compact open subgroup $\bfK_v\subset \bG(\mathfrak o_v)$
such that the above function equals 1 on $\bfK_v\times X(F_v)$.
Moreover, for almost all  $v$ the stabilizer of 
$(\cL,\|\cdot\|_v)$ is equal to $\bG(\mathfrak o_v)$.

We can use the same section $f$ for the right action of $\bG$. 
If $\cL$ is not ample, we can represent 
it as $\cL=\cL_1\otimes \cL_2^{-1}$
with very ample $\cL_1$, $\cL_2$ and apply the same argument.

Now choose a basis for $\Pic(X)$ consisting of very ample 
line bundles, fix smooth adelic 
bi-equivariant metrizations on these generators as described above
and extend these adelic metrizations to arbitrary $L\in\Pic(X)$
(by linearity). Then there exists a compact open subgroup
$$
\bfK=\prod_{v\in S_{\rm fin}}\bfK_v 
$$
stabilizing  {\em all} line bundles metrized in this way.   
\end{proof}

\subsection{Spectral analysis}
\label{sect:spe}

In this section we begin to analyze the spectral decomposition of the
height zeta function
\begin{equation}
\label{eqn:series}
\zZ({\bf s};g):=\sum_{\gamma\in \bG(F)} H({\bf s};\gamma g)^{-1}.
\end{equation}

\

\begin{prop}
\label{prop:abs-conv}
There exists an $N>0$ such that the series \eqref{eqn:series}
converges absolutely and uniformly to a holomorphic in ${\bf s}$ and continuous in $g$ function, 
for $g$ and ${\bf s}$ contained in compacts in
$\bG(\A)$ and, respectively, in 
the domain in $\Pic(X)_{\R}$ defined by $\Re(s_{\alpha})>N$, for all $\alpha$.
\end{prop}

\begin{proof}
Analogous to the proof of Proposition 4.4 in \cite{CLT}
(follows from the projectivity of $X$).
\end{proof}

\begin{prop}
\label{prop:formal}
One has a formal identity
\begin{equation}
\label{eqn:formal}
\zZ({\bf s};g)=\sum_{\varrho\in \cR(\bfK) } \zZ_{\varrho}({\bf s};g).
\end{equation}
Here the sum is over all irreducible unitary representations 
occurring  $\sL^2(\bG(F)\ba \bG(\A))^{\bfK}$, see~\ref{nota:RK}. 
\end{prop}

\begin{proof}
We use the right-$\bfK$-invariance of the height function.
\end{proof}

\begin{nota}
\label{nota:SX}
Denote by $S_X$ the set of all places $v\notin S_{\infty}$ such that
either:
\begin{itemize}
\item residual characteristic of $v$ is 2 or 3;
\item $\bfK_v\neq \bG(\mathfrak o_v)$;
\item $\vol(\bG(\mathfrak o_v))$ with respect to $dg_v$
is not equal to 1;
\item the coadjoint action of $\bG$ on $\mg^*$ is not defined
over $\mo_v$;
\item over $\mo_v$, the union $\cup_{\al}D_{\al}$ is not
a union of smooth relative divisors with strict normal crossings.
\end{itemize}
\end{nota}

\begin{rem}
\label{rem:invar}
For all $v\notin (S_X\cup S_{\infty})$ 
the height $H_v$ is invariant with respect to
the right and left $\bG(\mo_v)$-action. 
\end{rem}

\begin{nota}
\label{nota:SPi}
Denote by $S_{\varrho}$ the set of all  nonarchimedean places $v$ 
such that either
\begin{itemize}
\item $v\in S_X$;
\item $|\Pf(\ell)|_v\neq 1$, 
(where $\ell$ is contained in the orbit
corresponding to $\varrho$);
\item $v\in S_{\pi_\ell}$. 
\end{itemize}
\end{nota}

By the results in Section~\ref{sect:mult}, for all
$\varrho\in \cR(\bfK)$ 
and all nonarchimedean  $v\notin S_{\varrho}$ one has
$\dim\, \cH_{\varrho_v} =1$. For all other $v\notin S_{\infty}$,  
$\dim\, \cH_{\varrho_v}<\infty$, controlled by Proposition~\ref{prop:mult}.

\

Choose a norm 1 generator 
$$
\omega_{\varrho}=\otimes_{v\notin (S_{\varrho}\cup \infty)}\omega_{\varrho,v}
$$
of the 1-dimensional space
$$
\otimes_{v\notin (S\cup \infty)}\cH_{\varrho_v}^{\bfK_v}
$$
and an orthonormal basis $\cB_{S_{\varrho}}=\{ \omega_{\varrho,\la'}\}$ 
for the finite dimensional space 
$$
\otimes_{v\in S_{\varrho}} \cH_{\varrho_v}^{\bfK_v}.
$$
Fix an elliptic operator 
$\Delta$ acting on $\Gamma\ba \bG_{\infty}$
as in Section~\ref{sect:ellipt} (where $G_{\infty}$ and $\Gamma$ are defined in  Notation~\ref{nota:ginf}).
Finally, choose an orthonormal basis  
$\cB_{\varrho,\infty}:=\{ \omega_{\varrho,\la}\}$ of 
$$
\otimes_{v\in S_{\infty}} \cH_{\varrho_v} \subset \sL^2(\Gamma\ba\bG_{\infty})
$$
consisting of $\Delta$-eigenfunctions. 

\begin{lemm}
\label{lemm:basis}
For all $\varrho\in \cR(\bfK)$, the set 
$$
\cB_\varrho=\omega_{\varrho}\otimes \cB_{S_{\varrho}}\otimes\cB_{\varrho,\infty}
$$
is a complete orthonormal basis of $\cH_{\varrho}$.
\end{lemm}

From now on we fix a basis $\cB_\varrho$ 
as in Lemma~\ref{lemm:basis}. Considerations in Section~\ref{sect:orb-method}, in particular, 
Proposition~\ref{prop:spherical}, imply:

\begin{prop}
\label{prop:zzz} 
For every $\varrho\in \cR(\bfK)$ one has (formally)
$$
\zZ_{\varrho}({\bf s};g)= \zZ^{\varrho}({\bf s};g) \cdot 
\zZ_{S_{\varrho}}({\bf s};g) \cdot  
\zZ_{\varrho,\infty}({\bf s};g),
$$
where
$$
\zZ^{\varrho}({\bf s};g):=
\prod_{v\notin (S_{\varrho}\cup \infty)} \int_{\mathrm G(F_v)} H_v({\bf s};g_v'g_v)^{-1}
\bar{\omega}_{\varrho,v}(g_v'g_v)dg_v',
$$
$$
\zZ_{S_{\varrho}}({\bf s};g)=\sum_{\lambda'} 
\int_{\mathrm G(\mathbb A_{S_{\varrho}})}   H_{\mathbb A_{S_{\varrho}}}  ({\bf s};g'_{S_{\varrho}} g_{S_{\varrho}} )^{-1}
\overline{\omega}_{\lambda'}(g'_{S_{\varrho}}g_{S_{\varrho}}) dg'_{S_{\varrho}}, 
$$
and 
$$
\zZ_{\varrho,\infty}({\bf s};g)=\sum_{\la}\int_{\mathrm G_{\infty}} 
H_{\infty}({\bf s};g'_{\infty}g_{\infty})^{-1}
\ovl{\omega}_{\la}(g'_{\infty}g_{\infty})dg'_{\infty}.
$$
\end{prop}

The following sections justify this formal expansion, by
\begin{itemize}
\item analyzing 
the contributions from all places $v\notin (S_{\varrho}\cup \infty)$, 
\item estimating contributions for places $v\in S_{\varrho}$, and 
\item
deriving upper bounds at places $v\in S_{\infty}$, via integration by parts for suitable 
differential operators.
\end{itemize}

\section{Analytic properties of the height zeta function}
\label{sect:prop}

We have a general integrability result (see, e.g., Section 8 in \cite{CLT}):

\begin{lemm}
\label{lemm:integrable}
For all places $v$, the integral
$$
I_v({\bf s}):=\int_{\bG(F_v)} H_v({\bf s}; g_v)^{-1}dg_v
$$
is absolutely convergent to a holomorphic function in ${\bf s}$ in the domain 
$\Re(s_{\alpha})> \kappa_{\alpha}-1$, for all $\alpha$. 
For every $\eps>0$ there exists a constant $\mathsf c_v(\eps)$ such that 
$$
|I_v({\bf s})|<\mathsf c_v(\eps),  
$$
for all ${\bf s}$ with $\Re(s_{\alpha})> \kappa_{\alpha}-1 +\eps$, for all $\alpha$. 
\end{lemm}

We start by considering nonarchimedian places of good reduction and the contribution of the trivial representation 
to the height zeta function. 
We have a stratification of 
$X$ by locally closed subvarieties
$$
D_A^0:=D_A\setminus \cup_{A'\supsetneq A}D_{A'} 
$$
where 
$$
D_A:= \cap_{\al\in A}D_{\al}
$$
and $A\subset\cA$. 
In particular, $D_{\emptyset}=\bG$. 
A key result is the following computation:

\begin{prop}
\label{prop:trivial}
For all $v\notin S_X\cup S_{\infty}$ and all ${\bf s}$ with $\Re(s_{\alpha})>\kappa_{\alpha}-1$, for all $\alpha$, 
one has
$$
\int_{\bG(F_v)} H({\bf s};g_v)^{-1}dg_v =
q_v^{-n}\left( \sum_{A\subseteq \cA} D^0_{A}(\k_v)  
\prod_{\al\in A} \frac{q_v-1}{q_v^{s_{\al}-\kappa_{\al}+1}-1} \right).
$$
\end{prop}

\begin{proof}
This is Theorem 7.1 in \cite{CLT}. 
The proof proceeds as follows: 
for $v\notin S_X\cup S_{\infty}$ 
there is a {\em good} model $\cX$ of $X$ over $\mo_v$:
all boundary components $D_{\al}$ (and $\bG$) are defined over 
$\mo_v$ and form a strict normal crossing divisor. 
We can consider the reduction map 
$$
{\rm red}\,\,:\,\, 
X(F_v)=X(\mathfrak o_v)\ra X(\k_v)=
\sqcup_{A\subset \cA}\, D_A^0(\k_v).
$$
The main observation is that in a neighborhood of the preimage
in $X(F_v)$ of the point $\tilde{x}_v\subset D_A^0(\k_v)$ 
one can introduce local $v$-adic analytic coordinates 
$\{ x_{\al}\}_{\al=1,...,n}$ 
such that 
$$
H_v({\bf s};g)^{-1}=\prod_{\al \in A} |x_{\al}|_v^{s_{\al}}. 
$$
Now it suffices to keep track of the change of the measure $dg_v$:
$$
dg_v= \prod_{\al\notin A} dx_{\al} \cdot
\prod_{\al \in A} |x_{\al}|_v^{-\kappa_{\al}} dx_{\al},
$$
where $dx_{\al}$ are standard Haar measures on $F_v$. 
The obtained integrals over the maximal ideal $\mathfrak m_v\subset \mo_v$ 
are elementary
$$
\int_{{\rm red}^{-1}(\tilde{x}_v)} H_v({\bf s};g_v)^{-1}dg_v=
\prod_{\al\notin A}\int_{\mathfrak m_v} dx_{\al} \cdot
\prod_{\al \in A}\int_{\mathfrak m_v} 
q_v^{-(s_{\al}-\kappa_{\al})v(x_{\al})} dx_{\al}
$$
(see Theorem 9.1 in \cite{CLT}).
Summing over all $\tilde{x}_v\in X(\k_v)$ we obtain the claim.
\end{proof}

\begin{coro}
\label{coro:trivial}
Let 
$$
\zZ_0({\bf s}, g):=\int_{\rG(\A)}H({\bf s};g'g)^{-1}dg' 
$$
be the contribution of the trivial representation to the spectral expansion \eqref{eqn:formalz}.
Then 
$$
\mathbf s \mapsto \zZ_{0}({\bf s};g)
$$ 
is holomorphic for  $\Re(s_{\alpha}) >\kappa_\alpha$, for all $\alpha$, and continuous in $g$. 
Furthermore, the function
$$
\mathsf s\mapsto \prod_{\alpha\in\mathcal A}(s_{\alpha}-\kappa_{\alpha}) \cdot \zZ_{0}({\bf s};g)
$$
is holomorphic for $\Re(s_{\alpha}) > \kappa_{\alpha}-1/2$ and continuous in $g$, with 
$$
\lim_{{\bf s}\ra \kappa} \prod_{\alpha\in\mathcal A}(s_{\alpha}-\kappa_{\alpha})\cdot \zZ_0({\bf s};e) =
\tau(\mathcal K_X)\neq 0, 
$$
where $\tau(-\mathcal K_X)$ is the Tamagawa number defined in \cite{peyre}. 
\end{coro}

\begin{proof}
Apply Corollary 7.6 in \cite{CLT}.
\end{proof}

\no

\

The next step is to analyze contributions from nontrivial representations. We start by considering 
automorphic characters of $\bG$, i.e., 1-dimensional representations of $\bG(\A)$, trivial on $\bG(F)$. 
In the framework of the orbit method, 
these correspond to linear forms $\ell\in \mathfrak g^*$ which are trivial on $[\mathfrak g, \mathfrak g]$. 
The treatment of these is analogous to the one presented in Section 10 of \cite{CLT}. 
Let 
$$
\hat{H}({\bf s}; \psi_{\ell},g):=\int_{\bG(\A)} H({\bf s}; g'g)^{-1}\ovl{\psi}_{\ell}(g'g)dg'
$$
be the Fourier transform of the height function with respect to $\psi_{\ell}$. 
By Proposition~\ref{prop:trivial} and Corollary~\ref{coro:trivial}, 
this integral converges to a holomorpic function for 
$\Re(s_{\alpha}) >\kappa_{\alpha}$, for all $\alpha$. 
Let
$$
\zZ_1({\bf s};g):=\sum_{\ell\neq 0} \hat{H}({\bf s};\psi_{\ell},g), \quad 
$$
be the contribution of the set of all nontrivial
automorphic characters to the spectral expansion \eqref{eqn:formalz}.

\begin{prop}
\label{prop:characters}
The function 
$$
{\bf s}\mapsto \zZ_1({\bf s};g)
$$
is holomorphic for $\Re(s_{\alpha}) > \kappa_{\alpha}$, for all $\alpha$, and continuos in $g$. 
Furthermore, the function 
$$
{\bf s}\mapsto  \prod_{\alpha\in\mathcal A}(s_{\alpha}-\kappa_{\alpha}) \cdot \zZ_{1}({\bf s};g)
$$
is holomorphic for $\Re(s_{\alpha}) > \kappa_{\alpha} -1/2$, with 
$$
\lim_{{\bf s}\ra \kappa} \prod_{\alpha\in\mathcal A}(s_{\alpha}-\kappa_{\alpha})\cdot \zZ_1({\bf s};e) = 0.
$$
\end{prop}

\begin{proof}
The proof is identical to the proof of Theorem 6.3 in \cite{CLT}. 
We repeat the argument since it will be essential in the subsequent 
analysis of infinite-dimensional representations
$\varrho\in \mathcal R(\bfK)$.

The $\ell$ occurring in $\zZ_1$ are parametrized by a lattice $\mathfrak d$, minus $0$.  
For each $\ell$ we have a finite set of nonarchimedean places $S_{\ell}$, of ``bad'' reduction, defined as 
in Notation~\ref{nota:SPi}. Put $S:=S_{\ell}\cup S_{\infty}$.  The main steps of the proof are:
\begin{enumerate}
\item
In the domain $\Re(s_{\alpha}) >  \kappa_{\alpha}-1/2$,  
provide an upper bound of the form 
$$
|\prod_{v\in S_{\ell}} \hat{H}_v({\bf s}; \psi_{\ell},g)| \le \|\mathcal O_{\ell}\|_{\infty}^\eps, 
$$
for some $\eps>0$; here the norm $\|\mathcal O_{\ell}\|_{\infty}$, 
defined in Notation~\ref{nota:norm},  is equivalent to the 
euclidean norm of the linear form $\ell$ as the element of the 
lattice $\mathfrak d$. 
\item 
Establish meromorphic continuation of the Euler product
$$
\hat{H}^S({\bf s}; \psi_{\ell}, g):=\prod_{v\notin S} \hat{H}_v({\bf s}; \psi_{\ell}, g),
$$ 
of the form
$$
\hat{H}^S({\bf s}; \psi_{\ell}, g) = \prod_{\alpha\in \mathcal A(\ell)} \zeta_F(s_{\alpha}-\kappa_{\alpha}+1) \cdot 
\phi({\bf s}; \psi_{\ell}, g), 
$$
where $\mathcal A(\ell)\subsetneq \mathcal A$, and 
$\phi$ is a holomorphic function in the domain $\Re(s_{\alpha}) > \kappa_{\alpha}-1/2+\epsilon$, for all $\alpha$, 
satisfying a uniform bound
$$
|\phi({\bf s}; \psi_{\ell},g)|\le \mathsf c'\cdot (1+\|\Im({\bf s})\|)^{N'}\cdot 
\|\mathcal O_{\ell}\|_{\infty}^\eps,
$$
 for some constants $\eps$, $\mathsf c'$, and $N'$.  
\item In the domain  $\Re(s_{\alpha}) > \kappa_{\alpha}-1/2$, for all $\alpha$,  
and for any $N\in \mathbb N$, obtain upper bounds of the shape 
$$
|\hat{H}_{\infty}({\bf s};\psi_{\ell}, g)|\le \mathsf c'' \cdot
(1+\|\Im({\bf s}\|)^{N'} \cdot (1+\|\mathcal O_{\ell}\|_{\infty})^{-N}, 
$$ 
for some constants $\mathsf c''$ and $N'$, 
to insure convergence of the sum
$$
\sum_{\ell} \hat{H}^S({\bf s}; \psi_{\ell}, g)\cdot \hat{H}_{S_{\ell}}({\bf s}; \psi_{\ell}, g) \cdot 
\hat{H}_{S_{\infty}}({\bf s}; \psi_{\ell}, g).
$$
\end{enumerate}

This is the content of Section 10 of \cite{CLT}. We proceed to explain these steps in more detail.

\

Let $f$ be a polynomial function on $\mg$ and let
$$
{\rm div}(f)=E(f)-\sum_\alpha d_\alpha(f) D_\alpha,
$$
$$
\cA_0(f):=\{\alpha\, |\, d_\alpha(f)=0\}.
$$ 
Every nontrivial 
linear form $\ell\in \mathfrak g^*$ defines
a nontrivial rational function $f=f_{\ell}\in F(X)$, by 
$$
x\mapsto \langle \ell, \log(x)\rangle.
$$ 
We write $d_{\alpha}=d_{\alpha}(f)$ for the multiplicities of 
$f$ along the corresponding boundary strata $D_{\alpha}$. 

\

For $v\in S_{\ell}$, we replace $\psi_{\ell}$ by 1 and refer 
to Lemma 8.2 of \cite{CLT} (see also Lemma 4.1.1 in
\cite{CLT-igusa} for a general integrability result of this type). This gives (1). 
 
\

For $v\notin S_{\ell}\cup S_{\infty}$, we compute the integral 
defining $\hat H_v(\mathbf s;\psi_{\ell}, g)$ on residue classes, 
as in the proof of Proposition~\ref{prop:trivial}.
Let $\tilde x\in X(\k_v)$ and $A=\{\alpha\,;\, 
\tilde x\in D_{\alpha}\}$. There are three cases:

\

\emph{Case 1. $A=\emptyset$.} ---
Since $\psi_{\ell}$ is trivial on $\bG(\mo_v)$, for $v\notin S_{\ell}\cup S_{\infty}$, 
\[
\int_{\bG(\mo_v)} H_v({\bf s};g_v)^{-1}
     \ovl{\psi}_{\ell}(g_v) dg_v =1
\]

{\it Case 2.} 
$A=\{\alpha\}$ and $\tilde x\not\in E$. ---
We introduce $v$-adic 
analytic coordinates $x_{\alpha}$ and $y_{\beta}$
around $\tilde x$ such that locally
$$
f_{\ell}(x)=\langle \ell,\log(x)\rangle = u x_{\alpha}^{-d_{\alpha}}.
$$
Then
\begin{align*}
\int_{{\rm red}^{-1}(\tilde x)} &= 
   \int_{\mathfrak m_v\times\mathfrak m_v^{n-1}}
        q^{-(s_\alpha-\kappa_\alpha)v(x_\alpha)} 
        \ovl{\psi}(u x_{\alpha}^{-d_{\alpha}})
         \, dx_{\alpha}
         d\mathbf y \\
&= \frac{1}{q^{n-1}}
   \sum_{n_{\alpha}\geq 1}
            q^{- (1+s_{\alpha}-\kappa_{\alpha})n_{\alpha}}
            \int_{\mathfrak o_v^*}
              \psi(u \pi^{- n_{\alpha} d_{\alpha}} u_{\alpha}^{-d_{\alpha}})\,
                du_{\alpha},
\end{align*}
where the last integral is elementary (see Lemma 10.3 of \cite{CLT}).

{\it Case 3.} $\# A\geq 2$ or $\# A=1$ and $\tilde x\in E$. ---
Replacing $\psi$ by 1 we find that, for real ${\bf s}$,  
the contribution of these $\tilde{x}$ is bounded by
\[ 
\label{eq.ET}
\sum_{\# A\geq 2} \frac{\# D_A^\circ(\k_v)}{q^n} \prod_{\alpha\in A}
    \frac{q-1}{q^{1+s_\alpha-\kappa_\alpha}-1}
{}  + \sum_{A=\{\alpha\}} \frac{\#(D_\alpha\cap E)(\k_v)}{q^n}
         \frac{q-1}{q^{1+s_\alpha-\kappa_\alpha}-1}. 
\] 

Combining the calculations of the cases above, we obtain
\[ 
 \hat H_v(\mathbf s;\psi_{\ell}, g)
= 1+ \sum_{\alpha\in \cA_0(f)}
     \frac{\# D_\alpha^\circ (\k_v)}{q^n} \frac{q-1}{q^{1+s_\alpha-\kappa_\alpha}-1} 
 +  \mathit{ET} 
\] 
with ``an error term''~$\mathit{ET}$ on the order of $\mathrm O(q^{-(1+\delta)})$, for some $\delta=\delta(\epsilon)$, when 
$\Re(s_{\alpha}) > \kappa_{\alpha}-1/2+\epsilon$, for all $\alpha$.  
This implies (2).

\

For $v\in S_{\infty}$ we use integration by parts with respect to suitable vector fields following the 
proof of Proposition 8.4 of \cite{CLT}; this proves (3).

\medskip

As in the proof of Proposition 10.2 of \cite{CLT}, we deduce from these estimates that 
$\hat H(\mathbf s;\psi_\ell,g)$ has a meromorphic
continuation to the domain $\Re(s_{\alpha})>\kappa_{\alpha}-1/2,$ for all $\alpha$:
\begin{equation}
\label{eqn:reff} 
\hat H(\mathbf s;\psi_\ell,g)= \phi(\mathbf s;\psi_{\ell},g)
\prod_{\alpha\in \mathcal A_0(f_\ell)}
     \zeta_F(1+s_\alpha-\kappa_\alpha), 
\end{equation}
where $\phi$ is a holomorpic function in this domain and $\zeta_F$ is the Dedekind zeta function. 
Moreover, for any $N>0$ there exist  
constants $N'>0$ and $\mathsf c(\eps, N)$ such that
for any $\mathbf s\in\mathsf T_{-1/2+\epsilon}$,
one has the estimate
\[ |\phi(\mathbf s;\psi_\ell,g)| \leq \mathsf c(\epsilon,N) 
     (1+\|\Im(\mathbf s )|)^{N'}
     (1+\|\mathcal O_\ell\|)^{-N}. \]
This, in turn, implies the claimed properties of $\zZ_1({\bf s},g)$. 
\end{proof}

The proof of Proposition~\ref{prop:characters} (see Equation~\ref{eqn:reff}) and Corollary~\ref{coro:trivial} imply:

\begin{coro}
\label{coro:a-b-L}
Let $L=\sum_{\alpha\in \mathcal A} l_{\alpha}D_{\alpha}$, $l_{\alpha}>0$ for all $\alpha$. Then the function
$$
s\mapsto \hat H(sL;\psi_\ell,g)
$$
has the following properties:
\begin{itemize}
\item is holomorphic for $\Re(s)>a(L)$,
\item admits a meromorphic continuation to $\Re(s)<a(L)-\delta$, for some $\delta>0$, 
\item is holomorphic in this domain, except possibly at $s=a(L)$, 
where it could have a pole of multiplicity at most $b(L)$,
\item the multiplicity of the pole is strictly smaller when $\ell\neq 0$.  
\end{itemize}
\end{coro}

We generalize the argument above to the
infinite-dimensional representations 
occurring in the expansion \eqref{eqn:formalz} as follows:

\

Consider a stratification of the set of $\Ad^*(G)$-orbits in $\mg^*$ 
into finitely many affine strata $Z_{\sigma}\subset \mg^*$ as in Sections~\ref{sect:inva} and 
\ref{sect:para}; 
we may assume that for each $\sigma\in\Sigma$ we have a finite set of $\Ad^*(G)$-invariant polynomials 
$\{ P_{\sigma,j}\}\in F[\mg^*]$ separating the orbits (see Theorem~\ref{thm:quot} and 
Proposition~\ref{prop:strata}). 

Passing to a finer stratification, if necessary, we may assume that for each 
$\sigma$, we have a finite collection 
of polynomial functions $Q_{\sigma,i}\in F[\mg^*]$ 
defining the $F$-morphism
$$
\mathrm{pol}_{\sigma}: Z_{\sigma}\ra \Gr(k_{\sigma}, \mg)
$$
from Proposition~\ref{prop:ml-rat}, i.e., for each $\ell\in Z_{\sigma}$ the image 
$$
\mathfrak m_{\ell}:=\mathrm{pol}_{\sigma}(\ell)\subset \Gr(k_{\sigma}, \mg)
$$ 
is a polarizing subalgebra for $\ell$. The corresponding family of subgroups 
$\{M_{\ell}\}_{\ell\in Z_{\sigma}}$ defines an equidimensional family of equivariant compactifications
$$
M_{\ell}\subset Y_{\ell}\subset X,   
$$  
with boundaries
$$
D_{\ell}:=Y_{\ell}\setminus M_{\ell}. 
$$
Considerations in Section~\ref{sect:orb-method} imply 
$$
\mathcal R(\bfK) =\sqcup_{\sigma\in \Sigma}\, \mathcal R(\bfK)_{\sigma}
$$
We restrict the height zeta function to a 1-parameter function
$$
\zZ(s; g):=\zZ(-sK_X;g)
$$
and decompose
\begin{equation}
\label{eqn:sum1}
\zZ(s; g) = \zZ_0(s; g) + \zZ_1(s; g) +  \sum_{\sigma\in \Sigma'} \zZ_{\sigma}(s; g),
\end{equation}
where the sum is over packets of infinite-dimensional automorphic representations. 
We will establish the meromorphic properties of each term in this sum. 
The pole or highest order at $s=1$ will be provided only by the trivial representation, i.e., 
$\zZ_0(s; g)$. 

For each stratum $\sigma\in \Sigma'$, fix integral models over $\mo_{S_{\sigma}}$, 
for some finite set of nonarchimedean places $S_{\sigma}$, i.e., 
we assume that the polynomials $P_{\sigma, j}$ and $Q_{\sigma,i}$ have coefficients in $\mo_{S_{\sigma}}$. 
We may assume that $S_{\sigma}$ contains $S_X$. 
Then there is a sublattice $\mathfrak d_{\sigma}\subset \mg^*(F)$ of 
$\mathfrak o_{S_{\sigma}}$-integral points in $\mg^*(F)$ 
such that if $\varrho_{\ell}\in \cR(\bfK)_{\sigma}$ 
then $\ell\in Z_{\sigma}(\mathfrak d_{\sigma})= Z_{\sigma}(F)\cap \mathfrak d_{\sigma}$.  
We can refine the expansion \eqref{eqn:sum1}:
\begin{equation}
\label{eqn:sum2}
\zZ(s; g) =  \zZ_0(s; g) + \zZ_1(s; g) +  \sum_{\sigma\in \Sigma'} \sum_{\ell\in   Z_{\sigma}(\mathfrak d_{\sigma})} 
\zZ_{\varrho_\ell}(s; g), 
\end{equation}
with $\zZ_{\varrho_\ell} \in \mathcal H_{ell}$, as in Section~\ref{sect:spe}.

By Proposition~\ref{prop:zzz}, 
$$
\zZ_{\varrho_\ell}(s; g) = \zZ^{\varrho_\ell}(s; g) \cdot \zZ_{S_{\varrho_{\ell}}}(s; g) 
\cdot \zZ_{\rho_\ell, \infty}(s; g), 
$$
where, by Proposition~\ref{prop:spherical}
\begin{equation}
\label{eqn:zzm}
\zZ^{\varrho_\ell}(s; g) = 
\prod_{v\notin (S\cup \infty)} 
\int_{\mathrm M_{\ell}(F_v)} H_v(-K_X; h_vg_v)^{-s} \ovl{\psi}_{\ell}(h_vg_v)  d h_v,
\end{equation}
and the other factors are contributions from places in $S_{\varrho_{\ell}}$ (defined in Notation~\ref{nota:SPi}) and 
the places at infinity. The computation of the local intergrals in \eqref{eqn:zzm} 
is analogous to the one explained in the proof of 
Proposition~\ref{prop:characters}, except that we cannot guarantee that $Y_{\ell}$,  
the Zariski closure of $\mathrm M_{\ell}$ in $X$, is smooth, with normal crossing boundary. 
However, the height integral can be computed on a desingularization $\tilde{Y}_{\ell}$ of $Y_{\ell}$, constructed in 
Proposition~\ref{prop:uni-deg}. By Corollary ~\ref{coro:a-b-L}, the analytic properties of the function 
$$
s\mapsto \zZ_{\varrho_\ell}^S(s; g)
$$
are governed by the invariants
$$
(a(-K_X|_{Y_{\ell}}), b(-K_X|_{Y_{\ell}})),
$$
computed on the resolution $\tilde{Y}_{\ell}$ as in Section ~\ref{sect:invariants}.
Proposition~\ref{prop:m} insures that
\begin{equation}
\label{eqn:bbl}
\zZ^S_{\varrho_\ell}(s; g)=\frac{1}{(s-1)^{b-1}} \cdot \phi_{\ell}(s;g),
\end{equation}
where $\phi_{\ell}$ is holomorphic for $\Re(s)>1-\delta$, for some $\delta>0$, i.e.,
$\zZ^S_{\varrho_\ell}(s; g)$ admits a meromorphic continuation to this domainm, with a possible pole at $s=1$ of 
order {\em strictly smaller} than $b=b(-K_X)$, the rank of the Picard group of $X$.

Propositions~\ref{prop:m} and \ref{prop:uni-deg} provide uniform control on the geometry 
of the occurring desingularizations $Y_{\ell}$. In particular, only finitely many pairs 
$(a(-K_X|_{Y_{\ell}}), b(-K_X|_{Y_{\ell}}))$ arise, and 
the number and degrees of the corresponding boundary 
components are also uniformly bounded. The set of places of bad reduction of integral models of $Y_{\ell}$ 
is controlled by values of 
polynomials parametrizing orbits and defining 
the corresponding polarizing subalgebras in the stratum $\sigma$, i.e., by   
polynomial expressions in $P_{\sigma,j}(\ell)$. As in the proof of 
Proposition~\ref{prop:characters}, we obtain the bound
\begin{equation}
\label{eqn:bl}
|\phi_{\ell}(s;g)|\le \mathsf c\cdot |s| ^N\cdot \|\mathcal O_{\ell}\|_{\infty}^\eps, 
\end{equation}
for some constants $\mathsf c, N, \eps>0$, independent of $\ell$ and $g\in \mathrm G(F)\backslash 
\mathrm G(\A_F)$. 

By Proposition~\ref{prop:mult}, the sets $S_{\varrho_{\ell}}$ and the dimensions of $\mathcal B_{S_{\varrho_{\ell}}}$
are controlled in terms of   
an $\Ad^*(\rG)$-invariant 
polynomial ${\rm Pf}\in \mo_{S_{\sigma}}[\mg^*]$  (the Pfaffian). For $v\in S_{\varrho_{\ell}}$ we can use the trivial estimate, replacing $\omega_{\lambda'}$ by 1, and using the integrability of height functions as in  
Lemma 8.2 of \cite{CLT}, to obtain:

\begin{lemm}
\label{lemm:bad-red}
There exist $\delta,\eps, \mathsf c>0$ such that for all $\varrho_{\ell}\in \cR(\bfK)$
the function $\zZ_{S_{\varrho_{\ell}}}$ is holomorphic in the domain $\Re(s)>1-\delta$ and satisfies
\begin{equation}
\label{eqn:srhoell}
|\zZ_{S_{\varrho_{\ell}}}(s;g)|\le \mathsf c\cdot (1+\|\mathcal O_{\ell}\|_{\infty})^{\eps}.
\end{equation}
\end{lemm}

We now address contributions from archimedean places. 

\begin{lemm}
\label{lemm:extim}
Let $X$ be an equivariant 
compactification of a unipotent group  $\bG$ and 
${\mg}$ the Lie algebra of $\bG$.
For all $v\in S_{\infty}$, all $\eps>0$, and all
$\partial_v \in {\mathfrak g}(F_v)$  
there exist constants $\sc_v=\sc_v(\eps,\partial_v)$ and $N=N(\partial_v)\in \N$ such that
$$
\int_{\bG(F_v)}|\partial_v H_{v}({\bf s};g_v)^{-1}|_v d g_v < 
\sc_v\cdot \|{\bf s}\|^{N}
$$
for all ${\bf s}$ with $\Re(s_{\alpha}) >\kappa_\alpha -1 +\eps$, for all $\alpha$. 
\end{lemm} 

\begin{proof}
We use integration by parts with respect to $\partial_v$ as in 
the proof of Proposition 6.4 in \cite{CLT}.
Assume that $F_v=\R$. Let $x\in X(\R)$ and $A\subset \cA$ 
be the set of all $\al$ such that $x\in D_{\al}(\R)$.
Let $f_{\al}=0$ be a local equation for $D_{\al}$ in 
a neighborhood $U$ of $x$. Then there exist functions
$\varphi_{\al}\in \mathsf C^{\infty}(U)$ such that
for all $g\in U\cap \bG(\R)$ we have
$$
H_v({\bf s};g)^{-1}=\prod_{\al \in \cA} \exp(-s_{\al}h_{\al}(g))
$$
where 
$$
h_{\al}(g)=\log|f_{\al}(g)|  +\varphi_{\al}(g).
$$
Moreover, for any $\partial_v \in {\mathfrak g}(F_v)$ 
the derivative $\partial_v h_{\al}(g)$ 
extends to a $\mathsf C^{\infty}$-function on 
the compactification $X(\R)$ (compare Proposition 2.2 in \cite{CLT}).
In particular, it is bounded on $\bG(\R)\subset X(\R)$
and the claim follows. Complex places are treated in the same way. 
\end{proof}

Fix a stratum $\sigma$ and consider an $\ell\in\sigma$.
By Proposition~\ref{prop:scalar}, differential operators corresponding to the $\Ad(G)^*$-invariant polynomials $P_{\sigma,j}$, 
act in the representation space $\mathcal H_{\ell}$ 
by multiplication by $P_{\sigma,j}(2\pi i \ell)$. We apply Lemma~\ref{lemm:extim}:

\begin{coro}
\label{coro:estim}
For all $\eps>0$ and all $n\in \N$ there exists a constant $\sc=\sc(\eps,N)$ 
such that for all $\mathbf s$ with $\Re(s_{\alpha})>\kappa_{\alpha}-1+\eps$, for all $\alpha$, 
for all $\varrho_\ell\in \cR(\bfK)$, and all eigenfunctions 
$\omega_{\varrho_{\ell}, \la}$ as in Lemma~\ref{lemm:basis}
one has 
$$
|\int_{\bG_{\infty}}H_{\infty}({\bf s};g_{\infty})^{-1}\overline{\omega}_{\varrho_{\ell}, \la}(g_{\infty})dg_{\infty} |\le
\sc\cdot  \|{\bf s}\|^N\cdot \lambda^{-n} \cdot \|\mathcal O_\ell \|_{\infty}^{-n}
$$
\end{coro}

\begin{proof}
Proceed by intergration by parts, using the Laplacian  $\Delta=\Delta_{\ell}$ as in Section~\ref{sect:ellipt}.
For any $N\in \N$ and any  $\Delta$-eigenfunction 
$\omega_{\varrho_{\ell}, \lambda} \in \cB_{\varrho_\ell, \infty}$ with eigenvalue $\la$ we have
\begin{align*}
\la^N \cdot \int_{\bG_{\infty}} H_\infty({\bf s};g_\infty)^{-1}\bar\omega_{\varrho_{\ell}, \lambda} (g_\infty)dg_\infty & =   
\int_{\bG_{\infty}} H_\infty({\bf s};g_\infty)^{-1}\Delta^N\bar\omega_{\varrho_{\ell}, \lambda} (g_\infty)dg_\infty \\
& =  \int_{\bG_{\infty}}\Delta^N H_\infty({\bf s};g_\infty)^{-1}\bar\omega_{\varrho_{\ell}, \lambda} (g_\infty) dg_\infty,
\end{align*}
which is majorized by 
$$
\|\omega\|_{\L^{\infty}(\Gamma\backslash \bG_{\infty})}\cdot \int_{\bG_{\infty}}|\Delta^N 
H_\infty({\bf s};g_\infty)^{-1}|dg_\infty. 
$$
The third property in \ref{prop:ell-op} bounds the norm of $\omega_{\varrho_{\ell}, \lambda}$ and
Lemma~\ref{lemm:extim} the integral. 
Similarly, applying differential operators corresponding to $P_{\sigma, j}$ which act 
by $P_{\sigma,j}(2\pi i \ell)$ on the eigenfunctions, and using the 
definition of $\|\mathcal O_{\ell}\|_{\infty}$ 
we obtain the claim. 
\end{proof}

To establish analytic properties of the height zeta function we return to Equation~\ref{eqn:sum1} and 
consider
$$
\sum_{\sigma\in \Sigma'}  \zZ_{\sigma}(s;g).
$$
We have, formally, 
\begin{equation}
\label{eqn;ffo}
\zZ_{\sigma}(s;g)=\sum_{\ell\in \sigma} \zZ_{\varrho_{\ell}}(s;g). 
\end{equation}
By Proposition~\ref{prop:zzz}
$$ 
\zZ_{\varrho_{\ell}}(s;g)= \zZ^{\varrho_{\ell}}(s;g) \cdot \zZ_{S_{\varrho_{\ell}}}(s;g)\cdot \zZ_{\varrho_{\ell},\infty}(s;g).
$$
We combine Equations~\eqref{eqn:bbl}, \eqref{eqn:bl}, \eqref{eqn:srhoell} and Corollary~\ref{coro:estim}
to derive that there exists a $\delta>0$ such that for all $\ell\in \sigma$ we have
$$
\zZ_{\varrho_{\ell}}(s;g)=\frac{1}{(s-1)^{b-1}}\cdot \Phi_\ell(s;g),
$$
where $\Phi_{\ell}$ is holomorphic in $s$, for $\Re(s)>1-\delta$ and continuous in $g$. 
Moreover, for all $n\in \N$ there exist constants $\mathsf c=\mathsf c(\delta,N)$ and $N'$ such that 
$$
|\Phi_\ell(s;g)| \le \mathsf c \cdot |s|^{N'} \cdot 
\|\mathcal O_\ell\|_{\infty}^{-n},
$$ 
in this domain.

\section{Appendix: Elliptic operators}
\label{sect:ellipt}

Let $U\subset \R^n$ be an open subset and 
$$
\dD:=\sum_{|\rJ|\le m} 
f_{\rJ}(x)\left(-i\frac{\partial}{\partial x}\right)^{\rJ}
$$
a partial differential operator (we use 
the standard multi-index notations $\rJ=(j_1,...,j_n)$ etc).
Assume that $f_{\rJ}\in \sC^{\infty}(\R^n)$ for all $\rJ$. 
The principal symbol $\rP_{\dD}$ of $\dD$ is defined
as 
$$
\rP_{\dD}(x,\xi)=\sum_{|\rJ|=m}f_{\rJ}(x)\xi^{\rJ}
$$
(here $\xi=(\xi_1,...,\xi_n)\in \R^n$).
The operator $\dD$ is called {\em elliptic} in $U$ if 
for all $x\in U$ the equality $\rP_{\dD}(x,\xi)=0$ implies $\xi=0$. 

\

Let $M$ be a $\sC^{\infty}$-manifold equipped with a 
Riemannian metric 
and $\cT(M)$ the tangent bundle of $M$. 
Consider the maps 
$$
\begin{array}{ccc}
\cT(M)          & \ra     &  \End_{\C}(\sC^{\infty}(M))\\
       \partial & \mapsto   &  (f\mapsto \partial f)\\
\sC^{\infty}(M) & \ra     & \End_{\C}(\sC^{\infty}(M)) \\
g               & \mapsto &  (f\mapsto g\cdot f).
\end{array}
$$
The subalgebra $\tD(M)$ of $\End_{\C}(\sC^{\infty}(M))$ generated by
the above endomorphisms is called 
the algebra of (finite order) differential operators
of $M$. 

\begin{lemm}
\label{lemm:crit-ell}
Let $\Delta\in \tD(M)$ be an operator of the form 
$$
\Delta=\sum_j \partial_j^2,
$$ 
where 
$\partial_j\in \sC^{\infty}(\cT(M))$ (and 
$j$ runs over a finite set).

The operator $\Delta$ is elliptic iff there exists a constant $\sc>0$
such that for all $x\in M$ and all $\xi_x\in \cT_x^*(M)$ one
has
$$
\sum_{j} (\partial_j(x),\xi_x)^2 \ge \sc\cdot \|\xi_x\|^2,
$$
(where $\|\cdot\|$ is the Riemannian metric on $M$).
\end{lemm}

A crucial ingredient in the proof of analytic properties of
the height zeta function is the following basic fact about elliptic operators on compact
manifolds.

\begin{prop}
\label{prop:ell-op}
Let $M$ be a compact manifold and $\Delta$ an elliptic operator on 
$\sC^{\infty}(M)$. 
Then
\begin{itemize}
\item the set $\Spect(\Delta)$  
of eigenvalues of $\Delta$ is a discrete subset of $\R_{\ge 0}$;
\item there exists a constant $\sc_1>0$ 
such that the spectral zeta function 
$$
\sum_{\la\in \Spect(\Delta) \setminus 0} n_{\la}\lambda^{-s}
$$
converges absolutely and uniformly in 
compacts in the domain $\Re(s)>\sc_1$ (here $n_{\la}$ is 
the dimension of the $\la$-eigenspace); 
\item there exist constants $\sc_2,n>0$ such that   
for all $\lambda \in \Spect(\Delta)$ and all  
$\lambda$-eigenvectors $\omega$ one has the estimate
$$
\|\omega\|_{\sL^2(M)} \le \sc_2(1+\lambda^n) 
\|\omega\|_{\sL^{\infty}(M)}. 
$$ 
\end{itemize}
\end{prop}


\

We are interested in the case when $M=\Gamma\ba \bG_{\infty}$,
where $\bG_{\infty}$ is a $\sC^{\infty}$-Lie  
group and $\Gamma$ is a discrete cocompact subgroup. 
Denote by ${\mg}_{\infty}$ 
the Lie algebra of $\bG_{\infty}$ and by 
$$
\exp\,:\, {\mg}_{\infty}\ra \bG_{\infty}
$$
the exponential map.
Choose a basis $\ovl{\partial}_1,...,\ovl{\partial}_r$ 
of $\mathfrak g_{\infty}$.
Each $\ovl{\partial}_j$ can be regarded as a left-invariant 
vector field on $\bG_{\infty}$.

\begin{lemm}
\label{lemm:e}
The operator
$$
\ovl{\Delta}:=\sum_{j=1}^r \ovl{\partial}_j^2
$$
is an elliptic operator on $\bG_{\infty}$.
\end{lemm}

\begin{proof}
We may assume that the metric on $\bG_{\infty}$ 
is left invariant under the $\bG_{\infty}$-action. 
Thus it suffices to check 
the estimate from Lemma~\ref{lemm:crit-ell}
at the identity $e\in \bG_{\infty}$.
Choose a basis $\{ \ovl{\partial}^*_j\}$ of 
$\mathfrak g_{\infty}^*=\cT_e^*(\bG_{\infty})$
dual to $\{ \ovl{\partial}_j\}$ and write 
$$
\xi_e =\sum_{j=1}^r \xi_j \ovl{\partial}^*_j.
$$
Then 
$$
\sum_{j=1}^r (\ovl{\partial}_j(e),\xi_e)^2 
=\sum_{j=1}^r \xi_j^2 =\|\xi_e\|^2
$$ 
and we can take $\mathsf c=1$. 
\end{proof}

Consider the map
$$
\begin{array}{ccc}
\mg_{\infty} & \ra     &  \cT(M)\\
\ovl{\partial}       &  \mapsto & \partial=(f\mapsto 
\frac{d}{dt}|_{t=0}f(x\cdot \exp(t\partial))).
\end{array}
$$ 

\begin{lemm}
The operator
$$
\Delta:=\sum_{j=1}^r \partial_j^2
$$
is an elliptic operator on  $M$.
\end{lemm}

\begin{proof}
The manifold $M=\Gamma\ba \bG_{\infty}$ is 
locally isomorphic to $\bG_{\infty}$.
\end{proof}

\end{document}